\documentclass[11pt,a4paper,twoside]{article}

\usepackage[T1]{fontenc}
\usepackage[cp1250]{inputenc} 
\usepackage{times}

\usepackage[leqno]{amsmath} 
\usepackage{amsthm,amsfonts,amssymb} 
\usepackage{bm} 

\input xy  \xyoption{all}

\usepackage{geometry}
\usepackage{color}
\usepackage{enumerate} 
\usepackage{hyperref} 

\theoremstyle{plain}
\newtheorem{theorem}{Theorem}[section]
\newtheorem{lemma}[theorem]{Lemma}
\theoremstyle{definition}
\newtheorem{definition}[theorem]{Definition}
\newtheorem{example}[theorem]{Example}
\newtheorem{remmark}[theorem]{Remark}
\newtheorem{corollary}[theorem]{Corollary}

\newcommand{\reftext}[1]{#1}

\title{A note on actions of some monoids\footnote{This research was supported by  the  Polish National Science Centre grant under the contract number DEC-2012/06/A/ST1/00256. Published in Differential Geometry and Its Applications \textbf{47} (2016) pp. 212--245 (doi: 10.1016/j.difgeo.2016.04.003).
}}
\author{
Micha\l\ J\'{o}\'{z}wikowski\footnote{\emph{Institute of Mathematics. Polish Academy of Sciences. \'{S}niadeckich 8, 00-656 Warsaw, Poland} (email: \texttt{mjozwikowski@gmail.com})}\ \ and
Miko\l aj Rotkiewicz\footnote{\emph{Faculty of Mathematics, Informatics and Mechanics. University of Warsaw.  Banacha 2, 02-097 Warsaw, Poland} (email: \texttt{mrotkiew@mimuw.edu.pl})},
}
\date{\today}

\begin{document}
\maketitle
\begin{abstract} 
Smooth actions of the multiplicative monoid $(\mathbb{R},\cdot)$ of
real numbers on manifolds lead to an alternative, and for some reasons
simpler, definitions of a vector bundle, a double vector bundle and
related structures like a graded bundle (Grabowski and Rotkiewicz
(2011) \cite{JG_MR_gr_bund_hgm_str_2011}). For these reasons it is natural to study
smooth actions of certain monoids closely related with the monoid
$(\mathbb{R},\cdot)$. Namely, we discuss geometric structures
naturally related with: smooth and holomorphic actions of the monoid of
multiplicative complex numbers, smooth actions of the monoid of second
jets of punctured maps $(\mathbb{R},0)\rightarrow(\mathbb{R},0)$,
smooth actions of the monoid of real 2 by 2 matrices and smooth actions
of the multiplicative reals on a supermanifold. In particular cases we
recover the notions of a holomorphic vector bundle, a complex vector
bundle and a non-negatively graded manifold.
\end{abstract}
%
%
%
%



\section{Introduction}

\paragraph*{Motivation}
Our main motivation to undertake this study were the results of
Grabowski and Rotkiewicz \cite{JG_MR_higher_vec_bndls_and_multi_gr_sym_mnflds,JG_MR_gr_bund_hgm_str_2011}
concerning action of the multiplicative monoid of real numbers
$(\mathbb{R},\cdot)$ on smooth manifolds. In the first of the cited
papers the authors effectively characterized these smooth actions of
$(\mathbb{R},\cdot)$ on a manifold $M$ which come from homotheties of
a vector bundle structure on~$M$. In particular, it turned out that the
addition on a vector bundle is completely determined by the
multiplication by reals (yet the smoothness of this multiplication at
$0\in\mathbb{R}$ is essential). This, in turn, allowed for a
simplified and very elegant treatment of double and multiple vector bundles.

These considerations were further generalized in~\cite{JG_MR_gr_bund_hgm_str_2011}. The main result of that paper (here we
recall it as \reftext{Theorem~\ref{thm:eqiv_real}}) is an equivalence, in the
categorical sense, between smooth actions of $(\mathbb{R},\cdot)$ on
manifolds (\emph{homogeneity structures} in the
language of~\cite{JG_MR_gr_bund_hgm_str_2011}) and \emph{graded bundles}. Graded
bundles (introduced for the first time in \cite{JG_MR_gr_bund_hgm_str_2011}) can be viewed as a natural generalization
of vector bundles. In short, they are locally trivial fibered bundles
with fibers possessing a structure of a \emph{graded space}, i.e. a
manifold diffeomorphic to $\mathbb{R}^{n}$ with a distinguished class
of global coordinates with positive integer weights assigned. In a
special case when these weights are all equal to 1, a graded space
becomes a standard vector space and a graded bundle -- a vector bundle.

Surprisingly, graded bundles gained much more attention in
supergeometry, where they are called $N$-manifolds. One of the reasons
is that various important objects in mathematical physics can be seen
as $N$-manifolds equipped with an odd homological vector field. For
example, a Lie algebroid is a pair $(E, X)$ where $E$ is an
$N$-manifold of degree $1$, thus an anti-vector bundle, and $X$ is a
homological vector field on $E$ of weight~$1$. A~much deeper result
relates Courant algebroids and $N$-manifolds of degree $2$~\cite{Roytenberg}.

\paragraph*{Goals}
It is natural to ask about possible extensions of the results of
\cite{JG_MR_higher_vec_bndls_and_multi_gr_sym_mnflds,JG_MR_gr_bund_hgm_str_2011}
discussed above. There are two obvious directions of studies:
\begin{enumerate}[(Q1)]
\item[(Q1)] What are geometric structures naturally related with
smooth monoid actions on manifolds for monoids $\mathcal{G}$ other
than $(\mathbb{R},\cdot)$?
\item[(Q2)] How to characterize smooth actions of the multiplicative
reals $(\mathbb{R},\cdot)$ on supermanifolds?
\end{enumerate}
In this paper we provide answers to the above problems. Of course it is
hopeless to discuss (Q1) for an arbitrary monoid $\mathcal{G}$.
Therefore we concentrate our attention on several special cases, all
being natural generalizations of the monoid $(\mathbb{R},\cdot)$ of the
multiplicative reals:
\begin{itemize}
\item $\mathcal{G}=(\mathbb{C},\cdot)$ is the multiplicative monoid
of complex numbers;
\item $\mathcal{G}=\mathcal{G}_{2}$ is the monoid of the 2nd-jets
of punctured maps $\gamma:(\mathbb{R},0)\rightarrow
(\mathbb{R},0)$ (note that $(\mathbb{R},\cdot)$ can be viewed as the
monoid of the 1st-jets of such maps);
\item $\mathcal{G}=\operatorname{M}_{2}(\mathbb{R})$ is the monoid
of 2 by 2 real matrices.
\end{itemize}
Observe that all these examples contain $(\mathbb{R},\cdot)$ as a
submonoid. Therefore, by the results of \cite{JG_MR_gr_bund_hgm_str_2011}, every manifold with a smooth $\mathcal
{G}$-action will be canonically a graded bundle. This fact will be
often of crucial importance in our analysis.

\paragraph*{Main results} Below we list the most important results of
this paper regarding problem (Q1):
\begin{itemize}
\item For holomorphic actions of $\mathcal{G}=(\mathbb{C},\cdot)$ we
proved \reftext{Theorem~\ref{thm:equiv_holom}}, a direct analog of
\reftext{Theorem~\ref{thm:eqiv_real}}. Such actions (\emph{holomorphic homogeneity
structures} -- see \reftext{Definition~\ref{def:hgm_str_complex}}) are
equivalent (in the categorical sense) to \emph{holomorphic graded
bundles} (see \reftext{Definition~\ref{def:cplx_gr_bundle}}) -- a natural
extension of the notion of a (real) graded bundle to the holomorphic setting.
\item \reftext{Theorem~\ref{thm:equiv_complex}} is another analog of \reftext{Theorem~\ref{thm:eqiv_real}}. It characterizes \emph{complex graded bundles}
(defined analogously to the graded bundles in the real case -- see
\reftext{Definition~\ref{def:cplx_gr_bundle}}) in terms of \emph{complex
homogeneity structures} (i.e., smooth actions of $(\mathbb{C},\cdot)$
-- see \reftext{Definition~\ref{def:hgm_str_complex}}). It turns out that
complex graded bundles are equivalent (in the categorical sense) to a
special class of \emph{nice} complex homogeneity structures, i.e.
smooth $(\mathbb{C},\cdot)$-actions in which the imaginary part is in
a natural sense compatible with the action of $(\mathbb{R},\cdot
)\subset(\mathbb{C},\cdot)$ -- cf. \reftext{Definition~\ref{def:nice_hgm_str}}.
\item $\mathcal{G}_{2}$-actions on smooth manifolds are the main topic
of Section~\ref{sec:g2}. Since $\mathcal{G}_{2}$ is non-Abelian we
have to distinguish between left and right actions of $\mathcal
{G}_{2}$. As already mentioned, since $(\mathbb{R},\cdot)\subset
\mathcal{G}_{2}$, any manifold with $\mathcal{G}_{2}$-action is
naturally a (real) graded bundle. A crucial observation is that
$\mathcal{G}_{2}$ contains a group of additive reals $(\mathbb{R},+)$
as a submonoid. This fact allows to relate with every smooth right
(resp., left) $\mathcal{G}_{2}$-action a canonical complete vector
field (note that a smooth action of $(\mathbb{R},+)$ is a flow) of
weight $-$1 (resp., $+$1) with respect to the above-mentioned graded bundle
structure (\reftext{Lemma~\ref{lem:G2_actions_infinitesimally}}).

Unfortunately, the characterization of manifolds with a smooth $\mathcal
{G}_{2}$-action as graded bundles equipped with a weight $\pm1$
vector field is not complete. In general, such a data allows only to
define the action of the group of invertible elements ${\mathcal
{G}}^{\text{inv}}_{2}\subset\mathcal{G}_{2}$ on the considered
manifold (\reftext{Lemma~\ref{lem:G_2_two}}), still leaving open the problem of
extending such an action to the whole $\mathcal{G}_{2}$ in a smooth
way. For right $\mathcal{G}_{2}$-actions this question can be locally
answered for each particular case.
In \reftext{Lemma~\ref{lem:right_g2_deg_less_4}} we formulate such a result for graded bundles
of degree at most~3. The case of a left $\mathcal{G}_{2}$-action is
much more difficult and we were able to provide an answer (in an
elegant algebraic way) only for the case of graded bundles of degree
one (i.e., vector bundles) in \reftext{Lemma~\ref{lem:left_g2_action_vb}}.

\item As a natural application of our results about $\mathcal
{G}_{2}$-actions we were able to obtain a characterization of the
smooth actions of the monoid $\mathcal{G}$ of 2 by 2 real matrices in
\reftext{Lemma~\ref{lem:matrix}}. This is due to the fact that $\mathcal
{G}_{2}$ can be naturally embedded into $\mathcal{G}$. In the
considered case the action of $\mathcal{G}$ on a manifold provides it
with a double graded bundle structure together with a pair of vector
fields $X$ and $Y$ of bi-weights, respectively, $(-1,1)$ and $(1,-1)$
with respect to the bi-graded structure. Moreover, the commutator
$[X,Y]$ is related to the double graded structure on the manifold.
Unfortunately, this characterization suffers the same problems as the
one for $\mathcal{G}_{2}$-actions: not every structure of such type
comes from a $\mathcal{G}$-action.
\end{itemize}

Problem (Q2) is addressed in Section~\ref{sec:super} where we prove,
in \reftext{Theorem~\ref{thm:main_super}}, that supermanifolds equipped with a
smooth action of the monoid $(\mathbb{R}, \cdot)$ are \emph{graded
bundles in the category of supermanifold} in the sense of \reftext{Definition~\ref{def:s_grd_bndl}}
(the latter notion differs from the notion of an
$N$-manifold given in \cite{Roytenberg}: the parity of local
coordinates needs not to be equal to their weights modulo two). We are
aware that this result should be known to the experts. In \cite{Severa}
Severa states (without a proof) that ``an $N$-manifold (shorthand for
`non-negatively graded manifold') is a supermanifold with an action of
the multiplicative semigroup $(\mathbb{R}, \cdot)$ such that $-1$
acts as the parity operator'', which is a statement slightly weaker
than our result. Also recently we found a proof of a result similar to
Theorem~5.8 in \cite{BGG_grd_bndls_Lie_grpds},
Remark 2.2. Nevertheless, a rigorous proof of
\reftext{Theorem~\ref{thm:main_super}} seems to be missing in the literature (a version from
\cite{BGG_grd_bndls_Lie_grpds} is just a sketch). Therefore we decided
to provide it in this paper. It is worth to stress that our proof was
obtained completely independently to the one from \cite{BGG_grd_bndls_Lie_grpds}
and, unlike the latter, does not refer to the
proof of \reftext{Theorem~\ref{thm:eqiv_real}}.

\vspace*{-5pt}

\paragraph*{Literature}
Despite a vast literature on Lie theory for semi-groups (see e.g.,
\cite{HHL_Lie_semigroups,HH_Lie_semigroups}) we could not find
anything that deals with smooth actions of the monoid $(\mathbb{R},
\cdot)$ or its natural extensions. This can be caused by the fact that
the monoid $(\mathbb{R}, \cdot)$ is not embeddable to any group.

\vspace*{-5pt}

\paragraph*{Organization of the paper}
In Section~\ref{sec:perm} we briefly recall the main results of
 \cite{JG_MR_gr_bund_hgm_str_2011}, introducing (real) graded spaces, graded
bundles, homogeneity structures, as well as the related notions and
constructions. We also state \reftext{Theorem~\ref{thm:eqiv_real}} providing a
categorical equivalence between graded bundles and homogeneity
structures (i.e., smooth $(\mathbb{R},\cdot)$-actions). Later in this
section we introduce the monoid $\mathcal{G}_{2}$ and discuss its
basic properties.

Section~\ref{sec:complex} is devoted to the study of $(\mathbb
{C},\cdot)$-actions. Basing on analogous notions from Section~\ref{sec:perm} we define complex graded spaces, complex and holomorphic
graded bundles, as well as complex and holomorphic homogeneity
structures. Later we prove \reftext{Theorems~\ref{thm:equiv_holom} and \ref{thm:equiv_complex}} (these were already discussed above) providing
effective characterizations of holomorphic and complex graded bundles,
respectively, in terms of $(\mathbb{C},\cdot)$-actions.

The content of Sections~\ref{sec:g2} and \ref{sec:super} was
discussed in detail while presenting the main results of this paper.

\section{Preliminaries}
\label{sec:perm}
\paragraph*{Graded spaces}
We shall begin by introducing, after \cite{JG_MR_gr_bund_hgm_str_2011}, the notion of a (real) graded space.
Intuitively, a graded space is a manifold diffeomorphic to $\mathbb
{R}^{n}$ and equipped with an atlas of global graded coordinate
systems. That is, we choose coordinate functions with certain positive
integers (weights) assigned to them and consider transition functions
respecting these weights (they have to be polynomial in graded
coordinates). Thus a graded space can be understood as a natural
generalization of the notion of a vector space. Indeed, on a vector
space we can choose an atlas of global linear (weight one) coordinate
systems. Clearly, every passage between two such systems is realized by
a weight preserving (that is, linear) map. Below we provide a rigorous
definition of a graded space.

\begin{definition}\label{def:gr_space}
Let $\mathbf{d}= (d_{1}, \ldots, d_{k})$ be a sequence of
non-negative integers, let $I$ be a set of cardinality $|\mathbf
{d}|:=d_{1}+\ldots+d_{k}$, and let $I\ni\alpha\mapsto w^{\alpha}\in
\mathbb{Z}_{+}$ be a map such that $d_{i}= \#\{\alpha\in I: w^{\alpha}=i
\}$ for each $1\leq i\leq k$.

A \emph{graded space of rank $\mathbf{d}$} is a smooth manifold
$\mathrm{W}$ diffeomorphic to $\mathbb{R}^{|\mathbf{d}|}$ and
equipped with an equivalence class of graded coordinates. By
definition, a \emph{system of graded coordinates} on $\mathrm{W}$ is
a global coordinate system $(y^{a})_{a\in I}: \mathrm{W}\xrightarrow
{\simeq} \mathbb{R}^{|\mathbf{d}|}$ with \emph{weight} $w^{\alpha
}$ assigned to each function $y^{\alpha}$, $\alpha\in I$. To indicate
the presence of weights we shall sometimes write $y^{\alpha
}_{w^{\alpha}}$ instead of $y^{\alpha}$ and $\mathbf{w}(\alpha)$ to
denote the weight of $y^{\alpha}$.

Two systems of graded coordinates, $(y^{\alpha}_{w^{\alpha}})$ and
$(\underline{y}^{\alpha}_{\underline{w}^{\alpha}})$ are \emph
{equivalent} if there exist constants $c_{\alpha_{1} \ldots\alpha
_{j}}^{\alpha}\in\mathbb{R}$, defined for indices such that
$\underline{w}^{\alpha}=w^{\alpha_{1}} +\ldots+ w^{\alpha_{j}}$, satisfying
%
\begin{equation}\label{eqn:graded_transf_R}
\underline{y}^{\alpha}_{\underline{w}^{\alpha}} = \sum_{\substack
{j=1,2, \ldots\\ \underline{w}^{\alpha}=w^{\alpha_{1}}+\ldots
+w^{\alpha_{j}}}} c_{\alpha_{1} \ldots\alpha_{j}}^{\alpha
}y^{\alpha_{1}}_{w^{\alpha_{1}}} \ldots y^{\alpha_{k}}_{w^{\alpha_{k}}}.
\end{equation}
The highest coordinate weight (i.e., the highest number $i$ such that
$d_{i}\neq0$) is called the \emph{degree} of a graded space $\mathrm{W}$.

By a \emph{morphism between graded spaces} $\mathrm{W}_{1}$ and
$\mathrm{W}_{2}$ we understand a smooth map $\Phi:\mathrm
{W}_{1}\rightarrow\mathrm{W}_{2}$ which in some (and thus any) graded
coordinates writes as a polynomial homogeneous in weights $w^{\alpha}$.
\end{definition}

\begin{example} Consider a graded space $\mathrm{W}=\mathbb
{R}^{(2,1)}$ with coordinates $(x_{1},x_{2},y)$ of weights 1, 1 and 2,
respectively. A map $\Phi(x_{1},x_{2},y)=(3 x_{2},
x_{1}+2x_{2},y+x_{1}x_{2}+5(x_{2})^{2})$ is an automorphism of $\mathrm{W}$.
\end{example}

Observe that any graded space $\mathrm{W}$ induces an action
$h^{\mathrm{W}}: \mathbb{R}\times\mathrm{W}\rightarrow\mathrm{W}$
of the multiplicative monoid $(\mathbb{R}, \cdot)$ defined by
%
\begin{equation}
\label{eqn:hgm_structure}
h^{\mathrm{W}}(t, (y^{\alpha}_{w})) = (t^{w} \cdot y^{\alpha}_{w}).
\end{equation}
Indeed, it is straightforward to check that the formula for $h^{\mathrm
{W}}$ does not depend on the choice of graded coordinates $(y^{\alpha
}_{w})$ in a given equivalence class. We shall call $h^{\mathrm{W}}$
the action by \emph{homotheties} of $\mathrm{W}$. We will also use
notation $h^{\mathrm{W}}_{t}(\cdot)$ instead of $h^{\mathrm
{W}}(t,\cdot)$. The multiplicativity of $h^{\mathrm{W}}$ reads as
$h^{\mathrm{W}}_{s}\circ h^{\mathrm{W}}_{t}=h^{\mathrm{W}}_{t\cdot
s}$ for every $t,s\in\mathbb{R}$.

Obviously a morphism $\Phi:\mathrm{W}_{1}\rightarrow\mathrm{W}_{2}$
between two graded spaces intertwines the actions $h^{\mathrm{W}_{1}}$
and $h^{\mathrm{W}_{2}}$, that is
\[
h^{\mathrm{W}_{2}}_{t}(\Phi(v))=\Phi(h^{\mathrm{W}_{1}}_{t}(v))
\]
for every $v\in\mathrm{W}_{1}$ and every $t\in\mathbb{R}$. For
degree $1$ graded spaces we recover the notion of a linear map \cite
{JG_MR_higher_vec_bndls_and_multi_gr_sym_mnflds}.

\begin{example} The graded spaces $W_{1}= \mathbb{R}^{(1,0)}$ and
$W_{2}=\mathbb{R}^{(0,1)}$ are different although their underlying
manifolds are the same. Indeed, there is no diffeomorphism $f:\mathbb
{R}\to\mathbb{R}$ such that $f(t x) = t^{2} f(x)$ for every $t, x\in
\mathbb{R}$, that is one intertwining the associated actions of
$\mathbb{R}$ (cf. \reftext{Lemma~\ref{lem:hol_hmg_fun_R}}).
\end{example}

Using the above construction it is natural to introduce the following:
%
\begin{definition}\label{def:hgm_real}
A function $\phi: \mathrm{W}\rightarrow\mathbb{R}$ defined on a
graded space $\mathrm{W}$ is called \emph{homogeneous of weight} $w$
if for every $v\in\mathrm{W}$ and every $t\in\mathbb{R}$
%
\begin{equation}\label{eqn:def_hmg_fun_R}
\phi(h^{\mathrm{W}}(t, v)) = t^{w} \, \phi(v)\ .
\end{equation}

In a similar manner one can associate weights to other geometrical
objects on $\mathrm{W}$. For example a smooth vector field $X$ on $\mathrm
{W}$ is called \emph{homogeneous of weight $w$} if for every $v\in
\mathrm{W}$ and every $t>0$
%
\begin{equation}
\label{eqn:def_hgm_vf}
(h_{t})_{\ast}X(v)=t^{-w} X(h_{t}(v))\ .
\end{equation}
\end{definition}

We see that the coordinate functions $y^{\alpha}_{w}$ are functions of
weight $w$ and the field $\partial_{y^{\alpha}_{w}}$ is of weight
$-w$ in the sense of the above definition. In fact it can be proved that

\begin{lemma}[\cite{JG_MR_gr_bund_hgm_str_2011}]\label{lem:hol_hmg_fun_R}
Any homogeneous function on a graded space $\mathrm
{W}$ is a polynomial in the coordinate functions $y^{\alpha}_{w}$
homogeneous in weights~$w$.
\end{lemma}

\paragraph*{Graded bundles and homogeneity structures}
Since, as indicated above, a graded space can be seen as a
generalization of the notion of a vector space, it is natural to define
\emph{graded bundles} per analogy to vector bundles. A graded bundle
is just a fiber-bundle with the typical fiber being a graded space and
with transition maps respecting the graded space structure on fibers.

\begin{definition} A \emph{graded bundle of rank $\mathbf{d}$} is a
smooth fiber bundle $\tau: E\to M$ over a real smooth manifold $M$
with the typical fiber $\mathbb{R}^{\mathbf{d}}$ considered as a
graded space of rank $\mathbf{d}$. Equivalently, $\tau$ admits local
trivializations $\psi_{U}: \tau^{-1}(U)\to U \times\mathbb
{R}^{\mathbf{d}}$ such that transition maps $g_{UU'}(q):= \psi
_{U'}\circ\psi_{U}^{-1}|_{\{q\}\times\mathbb{R}^{\mathbf{d}}}:
\mathbb{R}^{\mathbf{d}}\to\mathbb{R}^{\mathbf{d}}$ are isomorphism
of graded spaces smoothly depending on $q\in U\cap U'$. By a \emph
{degree} of a graded bundle we shall understand the degree of the
typical fiber $\mathbb{R}^{\mathbf{d}}$.

A \emph{morphism of graded bundles} is defined as a fiber-bundle
morphism being a graded space morphism on fibers. Clearly, graded
bundles together with their morphisms form a \emph{category}.
\end{definition}

\begin{example}\label{ex:TkM}
A canonical example of a graded bundle is provided by the concept of a
\emph{higher tangent bundle}. Let $M$ be a smooth manifold and let
$\gamma,\delta:(-\varepsilon,\varepsilon)\rightarrow M$ be two
smooth curves on $M$. We say that $\gamma$ and $\delta$ have the same
$k$th-\emph{jet at} $0$ if, for every smooth function
$\phi:M\rightarrow\mathbb{R}$, the difference $\phi\circ\gamma
-\phi\circ\delta:(-\varepsilon,\varepsilon)\rightarrow\mathbb
{R}$ vanishes at $0$ up to order $k$. Equivalently, in any local
coordinate system on $M$ the Taylor expansions of $\gamma$ and $\delta
$ agree at $0$ up to order $k$. The $k$th-jet of $\gamma$
at $0$ shall be denoted by $\mathbf{t}^{k}\gamma(0)$. As a set the
$k$th-order tangent bundle $\mathrm{T}^{k}M$ consists of
all $k$th-jets of curves on $M$. It is naturally a bundle
over $M$ with the projection $\mathbf{t}^{k}\gamma(0)\mapsto\gamma
(0)$. It also has a natural structure of a smooth manifold and a graded
bundle of rank $(\underbrace{m,m,\ldots,m}_{k})$ with $m=\dim M$.
Indeed, given a local coordinate system $(x^{i})$ with $
i=1,\ldots, m$ on $M$ we define the so-called \emph{adapted
coordinate system} $(x^{i,(\alpha)})$ on $\mathrm{T}^{k}M$ with $i=
1,\ldots,m$, $\alpha=0,1,\ldots,k$ via the formula
\[
x^{i}(\gamma(t))=x^{i,(0)}(\mathbf{t}^{k}\gamma(0))+t\cdot
x^{i,(1)}(\mathbf{t}^{k}\gamma(0))+\ldots+\frac
{t^{k}}{k!}x^{i,(k)}(\mathbf{t}^{k}\gamma(0))+o(t^{k})\ .
\]
That is, $x^{i,(\alpha)}$ at $\mathbf{t}^{k}\gamma(0)$ is the
$\alpha$th-coefficient of the Taylor expansion of
$x^{i}(\gamma(t))$. We can assign weight $\alpha$ to the coordinate
$x^{i,(\alpha)}$. It is easy to check that a smooth change of local
coordinates on $M$ induces a change of the adapted coordinates on
$\mathrm{T}^{k}M$ which respects this grading.
\end{example}

Note that every graded bundle $\tau:E\rightarrow M$ induces a smooth
action $h^{E}:\mathbb{R}\times E\rightarrow E$ of the multiplicative
monoid $(\mathbb{R},\cdot)$ defined fiber-wise by the canonical
actions $h^{\mathrm{V}}$ given by \reftext{\eqref{eqn:hgm_structure}} with
$\mathrm{V}=\tau^{-1}(p)$ for $p\in M$. We shall refer to this as to
the action by \emph{homotheties of $E$}.

Clearly, $M=h^{E}_{0}(E)$ and any graded bundle morphisms $\Phi
:E_{1}\rightarrow E_{2}$ intertwine the actions $h^{E_{1}}$ and
$h^{E_{2}}$, i.e.,
\[
\Phi(h^{E_{1}}_{t}(e))=h^{E_{2}}_{t}(\Phi(e)),
\]
for every $e\in E_{1}$ and every $t\in\mathbb{R}$.

The above construction justifies the following:

\begin{definition}\label{def:hgm_str}
A \emph{homogeneity structure} on a manifold $E$ is a smooth action of
the multiplicative monoid $(\mathbb{R},\cdot)$
\[
h:\mathbb{R}\times E \longrightarrow E\ .
\]
A \emph{morphism} of two homogeneity structures $(E_{1},h^{1})$ and
$(E_{2},h^{2})$ is a smooth map $\Phi:E_{1}\rightarrow E_{2}$
intertwining the actions $h^{1}$ and $h^{2}$, i.e.,
\[
\Phi(h^{1}_{t}(e))=h^{2}_{t}(\Phi(e)),
\]
for every $e\in E_{1}$ and every $t\in\mathbb{R}$. Clearly,
homogeneity structures with their morphisms form a \emph{category}.

In the context of homogeneity structures we can also speak about \emph
{homogeneous functions} and \emph{homogeneous vector fields}. They are
defined analogously to the notions in \reftext{Definition~\ref{def:hgm_real}}.
\end{definition}

As we already observed graded bundles are naturally homogeneity
structures. The main result of \cite{JG_MR_gr_bund_hgm_str_2011}
states that the opposite is also true: there is an equivalence between
the category of graded bundles and the category of homogeneity
structures (when restricted to connected manifolds).

\begin{theorem}[\cite{JG_MR_gr_bund_hgm_str_2011}] \label{thm:eqiv_real}
The category of (connected) graded bundles is equivalent to the
category of (connected) homogeneity structures. At the level of objects
this equivalence is provided by the following two constructions:
\begin{itemize}
\item With every graded bundle $\tau:E\rightarrow M$ one can associate
the homogeneity structure $(E,h^{E})$, where $h^{E}$ is the action by
homotheties of $E$.
\item Given a homogeneity structure $(M,h)$, the map
$h_{0}:M\rightarrow M_{0}:=h_{0}(M)$ provides $M$ with a canonical
structure of a graded bundle such that $h$ is the related action by homotheties.
\end{itemize}
And at the level of morphism:
\begin{itemize}
\item Every graded bundle morphism $\Phi:E_{1}\rightarrow E_{2}$ is a
morphism of the related homogeneity structures $(E_{1},h^{E_{1}})$ and
$(E_{2},h^{E_{2}})$.
\item Every homogeneity structure morphism $\Phi
:(E_{1},h^{1})\rightarrow(E_{2},h^{2})$ is a morphism of graded
bundles $h^{1}_{0}:E_{1}\rightarrow h^{1}_{0}(M)$ and
$h^{2}_{0}:E_{2}\rightarrow h^{2}_{0}(M)$.
\end{itemize}
\end{theorem}

Let us comment briefly on the proof.
The passage from graded bundles to homogeneity structures is obtained
by considering the natural action by homotheties discussed above. The
crucial (and difficult) part of the proof is to show that for every
homogeneity structure $(M,h)$, the manifold $M$ has a graded bundle
structure over $h_{0}(M)$ compatible with the action $h$. The main idea
is to associate to every point $p\in M$ the $k$th-jet at
$t=0$ (for $k$ big enough) of the curve $t\mapsto h_{t}(p)$. In this
way we obtain an embedding $M\hookrightarrow\mathrm{T}^{k}M$, and the
graded bundle structure on $M$ can now be naturally induced from the
canonical graded bundle structure on $\mathrm{T}^{k} M$ (cf. Ex. \ref{ex:TkM}
and the proof of \reftext{Theorem~\ref{thm:equiv_holom}}).

The assumption of connectedness has just a technical character: we want
to prevent a situation when the fibers of a graded bundle over
different base components have different ranks.

\paragraph*{The weight vector field, the core and the natural affine
fibration}
Let us end the discussion of graded bundles by introducing several
constructions naturally associated with this notion.

Observe first that the homogeneity structure $h:\mathbb{R}\times
M\rightarrow M$ provides $M$ with a natural, globally-defined, action
of the additive group $(\mathbb{R},+)$ by the formula $(t,p)\mapsto
h_{e^{t}}(p)$. Clearly such an action is a flow of some (complete)
vector field.

\begin{definition}\label{def:euler_vf}
A (complete) vector field $\Delta_{M}$ on $M$ associated with the flow
$(t,p)\mapsto h_{e^{t}}(p)$ is called the \emph{weight vector field}
of $M$. Alternatively
$\Delta_{M}(p)=\frac{\mathrm{d}}{\mathrm{d}\,t}\big
|_{t=1}h_{t}(p)$.
\end{definition}

It is easy to show that, in local graded coordinates $(y^{\alpha
}_{w^{\alpha}})$ on $M$ (such coordinates exist since $M$ is a graded
bundle by \reftext{Theorem~\ref{thm:eqiv_real}}), the weight vector field reads as
\[
\Delta_{M}=\sum_{\alpha}w^{\alpha}y^{\alpha}_{w^{\alpha}}\partial
_{y^{\alpha}_{w^{\alpha}}}\ .
\]
Actually, specifying a weight vector field is equivalent to defining
the homogeneity structure on $M$. The passage from the weight vector
field to the action of $(\mathbb{R},\cdot)$ is given by
$h_{e^{t}}:=\exp(t\cdot\Delta_{M})$.

\begin{remmark}\label{rem:euler_weight}
The assignment $\phi\mapsto w\phi$ where $\phi:M\rightarrow\mathbb
{R}$ is a homogeneous function of weight $w$ can be extended to a
derivation in the algebra of smooth functions on $M$, thus a vector
field. Clearly, it coincides with the weight vector field $\Delta_{M}$
what justifies the name for $\Delta_{M}$.
Besides, the notion of a weight vector field can be used to study the
weights of certain geometrical objects defined on $M$ (cf. \reftext{Definitions~\ref{def:hgm_real} and \ref{def:hgm_str}}). This should be clear,
since the homogeneity structure used to define the weights can be
obtained by integrating the weight vector field.

For example, $X$ is a weight $w$ vector field on $M$ if and only if
\[
[\Delta_{M},X]=w\cdot X\ .
\]
Using the results of \reftext{Lemma~\ref{lem:hol_hmg_fun_R}}, it is easy to see
that, in local graded coordinates $(y^{\alpha}_{w^{\alpha}})$, such a
field has to be of the form
\[
X=\sum_{\alpha}P_{\alpha}\cdot\partial_{y^{\alpha}_{w^{\alpha
}}}\ ,
\]
where $P_{\alpha}$ is a homogeneous function of weight $w+w^{\alpha}$
for each index $\alpha$. Thus $X$, regarded as a derivation, takes a
function of weight $w'$ to a function of weight $w+w'$.
\end{remmark}

A graded bundle $\tau:E^{k}\rightarrow M$ of degree $k$ ($k\geq1$) is
fibrated by submanifolds defined (invariantly) by fixing values of all
coordinates of weight less or equal $j$ ($0\leq j\leq k$) in a given
graded coordinate system. The quotient space is a graded bundle of
degree $j$ equipped with an atlas inherited from the atlas of $E^{k}$
in an obvious way. The obtained bundles will be denoted by $\tau^{j}:
E^{j}\rightarrow M$. They can be put together into the following
sequence called \emph{the tower of affine bundle projections
associated with $E^{k}$}:
%
\begin{equation}\label{eqn:tower_grd_spaces}
E^{k}\xrightarrow{\tau^{k}_{k-1}} E^{k-1}\xrightarrow{\tau
^{k-1}_{k-2}} E^{k-2}\xrightarrow{\tau^{k-2}_{k-3}} \ldots
\xrightarrow{\tau^{2}_{1}} E^{1}\xrightarrow{\tau^{1}} M.
\end{equation}
Define (invariantly) a submanifold $\widehat{E^k}\subset E^{k}$ of a
graded bundle $\tau: E\rightarrow M$ of rank $(d_{1}, \ldots, d_{k})$
by setting to zero all fiber coordinates of degree less than $k$. It is
a graded subbundle of rank $(0,\ldots, 0, d_{k})$ but we shall
consider it as a vector bundle with homotheties $(t, (z^{\alpha}_{k}))
\mapsto(t \cdot z^{\alpha}_{k})$ and call it the \emph{core} of
$E^{k}$. It is worth to note that a morphism of graded bundles respects
the associated towers of affine bundle projections and induces a vector
bundle morphism on the core bundles.

\begin{example}\label{ex:TkM_tower_core} The core of $\mathrm{T}^{k}
M$ is $\mathrm{T}M$, while $\tau^{j}_{j-1}$ in the tower of affine
projections associated with $\mathrm{T}^{k} M$ is just the natural
projection to lower-order jets $\mathrm{T}^{j} M \rightarrow\mathrm
{T}^{j-1} M$.
\end{example}

\paragraph*{The monoid $\mathcal{G}_{k}$}
We shall end this introductory part by introducing $\mathcal{G}_{2}$,
the monoid of the 2nd-jets of punctured maps $\gamma
:(\mathbb{R},0)\rightarrow(\mathbb{R},0)$. Actually it is a special
case of
%
\begin{equation}\label{eqn:def_Gk}
\mathcal{G}_{k} := \{[\phi]_{k}\ |\ \phi:\mathbb{R}\rightarrow
\mathbb{R},\ \phi(0)=0\}\ ,
\end{equation}
the monoid of the $k$th-jets of punctured maps $\phi:
(\mathbb{R}, 0)\rightarrow(\mathbb{R}, 0)$. Here $[\phi]_{k}$
denotes an equivalence class of the relation
\[
\phi\sim_{k} \psi\quad\text{if and only if}\quad\phi
^{(j)}(0)=\psi^{(j)}(0)\quad\text{for every $j=1,2,\ldots, k$.}
\]
The natural multiplication on $\mathcal{G}_{k}$ is induced by the
composition of maps
\[
[\phi]_{k}\cdot[\psi]_{k}:=[\phi\circ\psi]_{k}\ ,
\]
for every $\phi,\psi:(\mathbb{R},0)\rightarrow(\mathbb{R},0)$.

Thus, a class $[\phi]_{k}$ is fully determined by the coefficients of
the Taylor expansion of $\phi$ at $0$ up to order $k$  and,
consequently, $\mathcal{G}_{k}$ can be seen as a set of polynomials of
degree less or equal $k$ vanishing at $0$ equipped with a natural
multiplication defined by composing polynomials and then truncating
terms of order greater than $k$:
\[
\mathcal{G}_{k} \simeq\{a_{1} t + \frac{1}{2} a_{2} t^{2} + \ldots+
\frac{1}{k!} a_{k} t^{k} + o(t^{k}): a_{1}, \ldots, a_{k}\in\mathbb
{R}\} \simeq\mathbb{R}^{k}.
\]

\begin{remmark}\label{rem:g_k_k=1_2}
Note that $\mathcal{G}_{1} \simeq(\mathbb{R}, \cdot)$ is just the
monoid of multiplicative reals, while the multiplication in $\mathcal
{G}_{2}\simeq\mathbb{R}^{2}$ is given by
\[
(a,b)(A,B)=(aA,a B+bA^{2})\ .
\]

Obviously,
\[
(\mathbb{R}, \cdot)\simeq\{[\phi]_{k}: \phi(t)=at, a\in\mathbb
{R}\}
\]
is a submonoid of $\mathcal{G}_{k}$ for every $k$.
\end{remmark}

Consider the set of algebra endomorphisms $\operatorname{End}(\mathbb
{D}^{k})$ of the Weil algebra $\mathbb{D}^{k}=\mathbb{R}[\varepsilon]/
\left\langle\varepsilon^{k+1}\right\rangle$. Every such an
endomorphism is uniquely determined by its value on the generator
$\varepsilon$, i.e., a map of the form
%
\begin{equation}
\label{eqn:basic_trans}
\varepsilon\longmapsto a_{1} \varepsilon+ \frac{1}{2}
a_{2}\varepsilon^{2} + \ldots+\frac{1}{k!} a_{k} \varepsilon^{k}.
\end{equation}
It is an automorphism if $a_{1}\neq0$.
Thus we may identify $\operatorname{End}(\mathbb{D}^{k})$ with
$\mathcal{G}_{k}$, taking into account that the multiplication
obtained by composing two endomorphisms of the form \reftext{\eqref{eqn:basic_trans}} is opposite to the product in $\mathcal{G}_{k}$
based on the identification \reftext{\eqref{eqn:def_Gk}}, i.e.
\[
\operatorname{End}(\mathbb{D}^{k})^{\text{op}} \simeq\mathcal{G}_{k}.
\]

\paragraph*{Left and right monoid actions}
Let $\mathcal{G}$ be an arbitrary monoid. By a \emph{left $\mathcal
{G}$-action} on a manifold $M$ we understand a map $\mathcal{G}\times
M\rightarrow M$ denoted by $(g, p)\mapsto g.p$ such that
$h.(g.p)=(h\cdot g).p$ for every $g,h\in\mathcal{G}$ and $p\in M$.
Here $h\cdot g$ denotes the multiplication in $\mathcal{G}$. \emph
{Right $\mathcal{G}$-actions} $M\times\mathcal{G}\rightarrow M$ are
defined analogously. Note that if the monoid multiplication is Abelian
(as is for example in the case of the multiplicative reals $(\mathbb
{R},\cdot)$ and the multiplicative complex numbers $(\mathbb{C},\cdot
)$), then every left action is automatically a right action and vice versa.

Note that any left $\mathcal{G}$-action $(g, p)\mapsto g.p$, gives
rise to a right $\mathcal{G}^{\text{op}}$-action on the same manifold
$M$ given by the formula $p.g = g.p$. However, unlike the case of
groups actions, in general, there is no canonical correspondence
between left and right $\mathcal{G}$-actions. For example, that is the
case if $\mathcal{G}=\mathcal{G}_{k}$ for $k\geq2$, since the
monoids $\mathcal{G}_{k}$ and $\mathcal{G}_{k}^{\text{op}}$ are not
isomorphic.

All monoids considered in our paper will be smooth, i.e., we restrict
our attention to monoids $\mathcal{G}$ which are smooth manifolds and
such that the multiplication $\cdot:\mathcal{G}\times\mathcal
{G}\rightarrow\mathcal{G}$ is a smooth map. We shall study smooth
actions, of these monoids on manifolds, i.e. the actions $\mathcal
{G}\times M\rightarrow M$ (or $M\times\mathcal{G}\rightarrow M$)
which are smooth maps.

We present now two canonical examples of left and right $\mathcal
{G}_{k}$-actions.
%
\begin{example}\label{ex:TkM_Gk} The natural composition of
$k$th-jets
\[
[\gamma]_{k}\circ[\phi]_{k} \mapsto[\gamma\circ\phi]_{k},
\]
for $\phi:\mathbb{R}\rightarrow\mathbb{R}$, $\gamma: \mathbb
{R}\rightarrow M$ where $\phi(0)=0$ defines a right $\mathcal
{G}_{k}$-action on the manifold $\mathrm{T}^{k} M$.
\end{example}

\begin{example}\label{ex:Gk_Tk_starM}
Following Tulczyjew's notation \cite{Tulczyjew_preprint}, the higher
cotangent space to a manifold $M$ at a point $p\in M$, denoted by
$\mathrm{T}_{p}^{k\ast} M$, consists of $k$th-jets at
$p\in M$ of functions $f: M\rightarrow\mathbb{R}$ such that $f(p)=0$.
The higher cotangent bundle $\mathrm{T}^{k\ast} M$ is a vector bundle
whose fibers are $\mathrm{T}_{p}^{k\ast} M$ for $p\in M$. The natural
composition of $k$th-jets
\[
[\phi]_{k}\circ[(f, p)]_{k} \mapsto[(\phi\circ f, p)]_{k},
\]
where $\phi:\mathbb{R}\rightarrow\mathbb{R}$ is as before, and $f:
M\rightarrow\mathbb{R}$ is such that $f(p)=0$ defines a left
$\mathcal{G}_{k}$-action on $\mathrm{T}^{k\ast}M$.
\end{example}

\section{On actions of the monoid of complex numbers}
\label{sec:complex}

In this section we study actions of the multiplicative monoid $(\mathbb
{C},\cdot)$ on manifolds. We begin by recalling some basic
construction from complex analysis, including the bundle of holomorphic
jets, and by introducing the notions of a complex graded space, per
analogy to \reftext{Definition~\ref{def:gr_space}} in the real case. Later,
again using the analogy with Section~\ref{sec:perm}, we define complex
and holomorphic graded bundles and homogeneity structures. The main
results of this section are contained in \reftext{Theorems~\ref{thm:equiv_holom} and \ref{thm:equiv_complex}}. The first states that
there is an equivalence between the categories of holomorphic
homogeneity structures (holomorphic action of the monoid $(\mathbb
{C},\cdot)$) and the category of holomorphic graded bundles. This
result and its proof is a clear analog of \reftext{Theorem~\ref{thm:eqiv_real}},
which deals with the smooth actions of the monoid $(\mathbb{R},\cdot
)$. In \reftext{Theorem~\ref{thm:equiv_complex}} we establish a similar
equivalence between smooth actions of $(\mathbb{C},\cdot)$ and
complex graded bundles. However, for this equivalence to hold we need
to input additional conditions (concerning, roughly speaking the
compatibility between the real and the complex parts of the action) on
the action of $(\mathbb{C},\cdot)$ in this case. We will call such
actions nice\ complex homogeneity structures (see \reftext{Definition~\ref{def:nice_hgm_str}}).

Throughout this section $N$ will be a complex manifold. We shall write
$N_{\mathbb{R}}$ for the smooth manifold associated with $N$, thus
$\dim_{\mathbb{R}}N_{\mathbb{R}}= 2\cdot\dim_{\mathbb{C}}N$. Let
$\mathcal{J}: \mathrm{T}N_{\mathbb{R}}\rightarrow\mathrm
{T}N_{\mathbb{R}}$ be the integrable almost complex structure defined
by~$N$.

\paragraph*{Holomorphic jet bundles} Following \cite
{Green_Griffiths_jet_bundles} we shall define $\mathrm{J}^{k}N$, the
space of $k$th-jets of holomorphic curves on $N$, and
recall its basic properties.

Consider a point $q\in N$, and let $\Delta_{R}\subset\mathbb{C}$
denote a disc of radius $0<R\leq\infty$ centered at $0$. Let
\[
\gamma:\Delta_{R}\rightarrow N
\]
be a holomorphic curve such that $\gamma(0)=q$. In a local holomorphic
coordinate system $(z^{j})$ around $q\in N$ the curve $\gamma$ is
given by a convergent series
\[
\gamma^{j}(\xi) := z^{j}(\gamma(\xi)) = a^{j}_{0} + a^{j}_{1} \xi+
a^{j}_{2} \xi^{2} + \ldots, \quad|\xi|<r
\]
where the coefficient $a^{j}_{l}$ equals $\frac{1}{l!} \frac{\mathbf
{d}^{l} \gamma^{j}}{\mathbf{d}\xi^{l}}(0)$ and $r\leq R$ is a
positive number. We say that curves $\gamma:\Delta_{R}\rightarrow N$,
and $\widetilde{\gamma}: \Delta_{\widetilde{R}}\rightarrow N$ \emph
{osculate to order $k$} if $\gamma(0)=\widetilde{\gamma}(0)$ and
$\frac{\mathbf{d}^{l} \gamma^{j}}{\mathbf{d}
\xi^{l}}(0)= \frac{\mathbf{d}^{l} \widetilde{\gamma}^{j}}{\mathbf
{d}\xi^{l}}(0)$ for any $j$ and $l=1,2, \ldots, k$. This property
does not depend on the choice of a holomorphic coordinate system around
$q$. The equivalence class of $\gamma$ will be denoted by $\mathbf
{j}^{k} \gamma(0)$ and called the $k$th-\emph{jet of a
holomorphic curve $\gamma$ at $q=\gamma(0)$}, while the set of all such
$k$th-jets at a given point $q$ will be denoted by
$\mathrm{J}^{k}_{q} N$. The totality
\[
\mathrm{J}^{k} N = \bigcup_{q\in M} \mathrm{J}^{k}_{q} N
\]
turns out to be a holomorphic (yet, in general, not linear) bundle over
$N$. Indeed, the local coordinate system $(z^{j})$ for $N$ gives rise
to a local \emph{adapted coordinate system} $(z^{j, (\alpha)})_{0\leq
\alpha\leq k}$ for $\mathrm{J}^{k} N$ where $z^{j, (\alpha)}(\mathbf
{j}^{k} \gamma(0)) = \frac{\mathbf{d}^{l} \gamma^{j}}{\mathbf
{d}\xi^{l}}(0)$. It is easy to see that transition functions between
two such adapted coordinate systems are holomorphic. The bundle
$\mathrm{J}^{k} N$ is called the \emph{bundle of holomorphic} $k$th
\textit{jets of}~$N$.

We note also that a complex curve $\gamma:\Delta_{R}\rightarrow N$
lifts naturally to a (holomorphic) curve $\mathbf{j}^{k} \gamma:
\Delta_{R}\rightarrow\mathrm{J}^{k} N$. Moreover, a holomorphic map
$\Phi: N_{1}\rightarrow N_{2}$ induces a (holomorphic) map $\mathrm
{J}^{k} \Phi: \mathrm{J}^{k} N_{1}\rightarrow\mathrm{J}^{k} N_{2}$
defined analogously as in the category of smooth manifolds, i.e.,
$\mathrm{J}^{k} \Phi(\mathbf{j}^{k} \gamma):= \mathbf{j}^{k}(\Phi
\circ\gamma)$. This construction makes $\mathrm{J}^{k}$ a functor
from the category of complex manifolds into the category of holomorphic bundles.

Finally, let us comment that we do not work with any fixed radius $R$
of the disc $\Delta_{R}$, being the domain of the holomorphic curve
$\gamma$. Instead, by letting $R$ be an arbitrary positive number, we
work with germs of holomorphic curves. This allows to avoid technical
problems related with the notion of holomorphicity.
For example, according to Liouville's theorem, any holomorphic function
$\phi:\Delta_{R}\rightarrow\mathbb{C}$ satisfies
\[
|\phi^{(j)}(0)| \leq\frac{j!}{R^{j}} \operatorname{sup}_{\xi\in
\Delta_{R}} |\phi(\xi)|.
\]
It follows that, in general for a fixed radius $R$, coefficients
$a^{j}_{i}$ cannot be arbitrary.
In particular, it may happen that there are no holomorphic curves
$\gamma: \Delta_{\infty}\rightarrow M$ except for constant ones, or
that for a fixed value $R<\infty$ there are no holomorphic curves
$\gamma:\Delta_{r}\rightarrow M$ such that the derivatives $\frac
{\mathbf{d}^{l}\gamma^{j}}{\mathbf{d}\xi^{l}}(0)$ have given
earlier prescribed values (see for the concept of hyperbolicity and
Kobayashi metric \cite{Kobayashi_book,Ch_Wong_Finsler_geom_hol_jet_bndls}). \medskip

It is a well-known fact from analytic function theory, that a
holomorphic map $\gamma:\Delta_{R}\rightarrow N$ is uniquely
determined by its restriction to the real line $\mathbb{R}\subset
\mathbb{C}$. Therefore it should not be surprising that the
holomorphic jet bundle $\mathrm{J}^{k} N$ can be canonically
identified with the higher tangent bundle $\mathrm{T}^{k} N_{\mathbb
{R}}$. The latter is naturally equipped with an almost complex
structure $\mathcal{J}^{k}: \mathrm{T}\mathrm{T}^{k} N_{\mathbb
{R}}\rightarrow\mathrm{T}\mathrm{T}^{k} N_{\mathbb{R}}$ induced
from $\mathcal{J}$, the almost complex structure on $N_{\mathbb{R}}$.
Namely, $\mathcal{J}^{k}:= \kappa_{k}\circ\mathrm{T}^{k} \mathcal
{J}\circ\kappa_{k}^{-1}$, where $\kappa_{k}: \mathrm{T}^{k} \mathrm
{T}N_{\mathbb{R}}\rightarrow\mathrm{T}\mathrm{T}^{k} N_{\mathbb
{R}}$ is the canonical flip. In local adapted coordinates $(x^{j,
(\alpha)}, y^{j, (\alpha)})$ on $\mathrm{T}^{k} N_{\mathbb{R}}$
such that $z^{j} = x^{j} + \sqrt{-1} y^{j}$ we have
\[
\mathcal{J}^{k} \left( \frac{\partial}{\partial x^{j, (\alpha
)}}\right) = \frac{\partial}{\partial y^{j, (\alpha)}}, \quad
\mathcal{J}^{k}\left( \frac{\partial}{\partial y^{j, (\alpha
)}}\right) = - \frac{\partial}{\partial x^{j, (\alpha)}}.
\]
Thus the local functions $z^{j}_{\alpha}:= x^{j, (\alpha)} + \sqrt
{-1} y^{j, (\alpha)}$ form a system of local complex coordinates on
$\mathrm{T}^{k} N_{\mathbb{R}}$, and hence we may treat $\mathrm
{T}^{k} N_{\mathbb{R}}$ as a complex manifold. We identify $\mathrm
{J}^{k} N$ with $\mathrm{T}^{k} N_{\mathbb{R}}$ as complex manifolds,
by means of the map $\mathbf{j}^{k} \gamma\mapsto\mathbf{t}^{k}
\gamma_{|\mathbb{R}}$, where $\gamma_{|\mathbb{R}}$ is the
restriction of $\gamma$ to the real line $\mathbb{R}\subset\mathbb
{C}$. In local coordinates this identification looks rather trivially:
$(z^{j, (\alpha)})\mapsto(z^{j}_{\alpha})$. This construction is
clearly functorial, i.e., for any holomorphic map $\Phi:
N_{1}\rightarrow N_{2}$ the following diagram commutes:
\[
\xymatrix{
\mathrm{J}^k N_1 \ar[d]^{\simeq}\ar[rr]^{\mathrm{J}^k \Phi} &&
\mathrm{J}^k N_2\ar[d]^{\simeq} \\
\mathrm{T}^k (N_1)_{\mathbb{R}} \ar[rr]^{\mathrm{T}^k \Phi} &&
\mathrm{T}^k (N_2)_{\mathbb{R}}.
}
\]

For $k=1$ there is another canonical identification of the real tangent
bundle $\mathrm{T}N_{\mathbb{R}}$ with the, so-called, \emph
{holomorphic tangent bundle} of $N$. Consider, namely, the
complexification $\mathrm{T}^{\mathbb{C}}N:= \mathrm{T}N_{\mathbb
{R}}\otimes\mathbb{C}$ and extend $\mathcal{J}$ to a $\mathbb
{C}$-linear endomorphism $\mathcal{J}^{\mathbb{C}}$ of $\mathrm
{T}^{\mathbb{C}}N$. The $(+i)$- and $(-i)$-eigenspaces of $\mathcal
{J}^{\mathbb{C}}$ define the canonical decomposition
\[
\mathrm{T}^{\mathbb{C}}N = \mathrm{T}'N \oplus\mathrm{T}'' N
\]
of $\mathrm{T}^{\mathbb{C}}N$ into the direct sum of complex
subbundles $\mathrm{T}' N$, $\mathrm{T}'' N$ called, respectively,
\emph{holomorphic} and \emph{anti-holomorphic tangent bundles of
$N$}. It is easy to see, that the composition $\mathrm{T}N_{\mathbb
{R}}\subset\mathrm{T}^{\mathbb{C}}N \twoheadrightarrow\mathrm
{T}'N$ gives a complex bundle isomorphism $\mathrm{T}N_{\mathbb
{R}}\simeq\mathrm{T}' N$. Let $\Phi: (N_{1})_{\mathbb
{R}}\rightarrow(N_{2})_{\mathbb{R}}$ be a smooth map and denote by
$\mathrm{T}^{\mathbb{C}}\Phi: \mathrm{T}^{\mathbb{C}}N_{1}
\rightarrow\mathrm{T}^{\mathbb{C}}N_{2}$ the $\mathbb{C}$-linear
extension of $\mathrm{T}\Phi$. It is well known that $\Phi$ is
holomorphic if and only if $\mathrm{T}\Phi$ is $\mathbb{C}$-linear.
The latter is equivalent to $\mathrm{T}^{\mathbb{C}}\Phi(\mathrm
{T}'N_{1})\subset\mathrm{T}' N_{2}$. In such a case we denote
$\mathrm{T}' \Phi: = \mathrm{T}^{\mathbb{C}}\Phi|_{\mathrm{T}'
N_{1}}: \mathrm{T}' N_{1}\rightarrow\mathrm{T}' N_{2}$.

Thus, under the canonical identifications discussed above, all three
constructions $\mathrm{J}^{1} \Phi$, $\mathrm{T}\Phi$ and $\mathrm
{T}'\Phi$ coincide (although the functor $\mathrm{T}$ is applicable
to a wider class of maps than $\mathrm{J}^{1}$ and $\mathrm{T}'$).

In what follows, given a smooth map $\phi: N\rightarrow\mathbb{C}$,
the real differential of $\phi$ at point $q\in N$ is denoted by
$\mathbf{d}_{q} \phi: \mathrm{T}N_{\mathbb{R}}\rightarrow\mathbb
{C}$. If $\phi$ happens to be holomorphic, then $\mathbf{d}_{q}\phi$
is $\mathbb{C}$-linear.

\paragraph*{Holomorphic and complex graded bundles}
A notion of a graded space has its obvious complex counterpart. It is
a generalization of a complex vector space. We basically rewrite the
definitions from Section~\ref{sec:perm} in the holomorphic context.

\begin{definition}
Let $\mathbf{d}= (d_{1}, \ldots, d_{k})$ be a sequence of
non-negative integers, let $I$ be a set of cardinality $|\mathbf
{d}|:=d_{1}+\ldots+d_{k}$, and let $I\ni\alpha\mapsto w^{\alpha}\in
\mathbb{Z}_{+}$, be a map such that $d_{i}= \#\{\alpha\in I: w^{a}=i \}$
for each $1\leq i\leq k$.

A \emph{complex graded space of rank $\mathbf{d}$} is a complex
manifold $\mathrm{V}$ biholomorphic with $\mathbb{C}^{|\mathbf{d}|}$
and equipped with an equivalence class of complex graded coordinates.
By definition, a \emph{system of complex graded coordinates} on
$\mathrm{V}$ is a global complex coordinate system $(z^{\alpha
})_{\alpha\in I}: \mathrm{V}\overset{\simeq}{\rightarrow}\mathbb
{C}^{|\mathbf{d}|}$ with \emph{weight} $w^{\alpha}$ assigned to each
function $z^{\alpha}$, $\alpha\in I$. To indicate the presence of
weights we shall sometimes write $z^{\alpha}_{w^{\alpha}}$ instead of
$z^{\alpha}$.

Two systems of complex graded coordinates, $(z^{\alpha}_{w^{\alpha
}})$ and $(\underline{z}^{\alpha}_{\underline{w}^{\alpha}})$ are
\emph{equivalent} if there exist constants $c_{\alpha_{1} \ldots
\alpha_{j}}^{\alpha}\in\mathbb{C}$, defined for indices such that
$\underline{w}^{\alpha}=w^{\alpha_{1}} +\ldots+ w^{\alpha_{j}}$,
satisfying
%
\begin{equation}\label{eqn:graded_transf}
\underline{z}^{\alpha}_{\underline{w}^{\alpha}} = \sum_{\substack
{j=1,2, \ldots\\ \underline{w}^{\alpha}=w^{\alpha_{1}}+\ldots
+w^{\alpha_{j}}}} c_{\alpha_{1} \ldots\alpha_{j}}^{\alpha
}z^{\alpha_{1}}_{w^{\alpha_{1}}} \ldots z^{\alpha_{k}}_{w^{\alpha_{j}}}.
\end{equation}
The highest coordinate weight (i.e., the highest number $i$ such that
$d_{i}\neq0$) is called the \emph{degree} of a complex graded space
$\mathrm{V}$.

By a \emph{morphism between complex graded spaces} $\mathrm{V}_{1}$
and $\mathrm{V}_{2}$ we understand a holomorphic map $\Phi:\mathrm
{V}_{1}\rightarrow\mathrm{V}_{2}$ which in some (and thus any)
complex graded coordinates writes as a polynomial homogeneous in
weights~$w^{\alpha}$.
\end{definition}

We remark that weights $(w^{\alpha})$ which are assigned to
coordinates on $\mathrm{V}$ are a part of the structure of the graded
space $\mathrm{V}$. Note that the set of functions $(\underline
{z}^{\alpha}_{\underline{w}^{\alpha}})$ defined by \reftext{\eqref{eqn:graded_transf}} defines a biholomorphic map if and only if the
matrices $(c^{\alpha}_{\beta})_{j}$ with fixed weights $w^{\alpha
}=w^{\beta}= j$ are non-singular for $j=1, 2, \ldots, k$.

\begin{remmark}
A complex graded space of rank $\mathbf{d}=(k)$ is just a complex
vector space of dimension $k$. Indeed, there is no possibility to
define a smooth map (a hypothetical addition of vectors in $\mathbb
{C}^{k}$) $+': \mathbb{C}^{k}\times\mathbb{C}^{k} \to\mathbb
{C}^{k}$ different from the standard addition of vectors in $\mathbb
{C}^{k}$ and such that $\mathbb{C}^{k}$ equipped with the standard
multiplication by complex numbers and addition $+'$ would satisfy all
axioms of a complex vector space.
This follows immediately from an analogous statement for a real vector
space: $+'$ would define an alternative addition on $\mathbb
{R}^{2n}\approx\mathbb{C}^{n}$, which is impossible.
\end{remmark}

Analogously to the real case, any complex graded space induces
an action $h^{\mathrm{V}}: \mathbb{C}\times\mathrm{V}\to\mathrm
{V}$ of the multiplicative monoid $(\mathbb{C}, \cdot)$ defined in
complex graded coordinates $(z^\alpha_w)$ by
\[
h^{\mathrm{V}}(\xi, (z^{\alpha}_{w})) = (\xi^{w} \cdot z^{\alpha}_{w}).
\]
Here $\xi\in\mathbb{C}$. We shall call $h^{\mathrm{V}}$ the action
by \emph{homotheties} of $\mathrm{V}$. Instead of $h^{\mathrm
{V}}(\xi,\cdot)$ we shall also write $h^{\mathrm{V}}_{\xi}(\cdot)$.

\begin{definition} A smooth function $f: \mathrm{V}\rightarrow\mathbb
{C}$ defined on a graded space $\mathrm{V}$ is called \emph{complex
homogeneous of weight} $w$ if
%
\begin{equation}\label{eqn:def_hmg_fun}
f(h^{\mathrm{V}}(\xi, v)) = \xi^{w} \, f(v)
\end{equation}
for any $v\in\mathrm{V}$ and any $\xi\in\mathbb{C}$.
\end{definition}

We see that the coordinate functions $z^{\alpha}_{w}$ are of
weight $w$ in the sense of above definition.
%
\begin{lemma}\label{lem:hol_hmg_fun}
Any complex homogeneous function on a complex graded space $\mathrm{V}$
is a polynomial in the coordinate functions $z^{\alpha}_{w}$
homogeneous in weights $w$.
\end{lemma}
\begin{proof}
Let $\phi:\mathrm{V}\rightarrow\mathbb{C}$ be a complex homogeneous
function of weight $w$. Note that $\mathrm{V}$ can be canonically
treated as a (real) graded space (of rank $2\mathbf{d}$, where
$\mathbf{d}$ is the rank of $\mathrm{V}$), simply by assigning
weights $w^{\alpha}$ to real coordinates $x^{\alpha}$, $y^{\alpha}$
such that $z^{\alpha}= x^{\alpha}+ \sqrt{-1} \, y^{\alpha}$ are
complex graded coordinates of weight $w^{\alpha}$. Let us denote this
graded space by $\mathrm{V}_{\mathbb{R}}$. Clearly, the real and
imaginary parts of $\Re\phi$, $\Im\phi:\mathrm{V}_{\mathbb
{R}}\rightarrow\mathbb{R}$ are of weight $w$, in the sense of
\reftext{Definition~\ref{def:hgm_real}}, with respect to this graded structure.
Thus, in view of \reftext{Lemma~\ref{lem:hol_hmg_fun_R}}, $\Re\phi$ and $\Im
\phi$ are real polynomials (homogeneous of weight $w$) in $x^{\alpha
}$ and $y^{\alpha}$. We conclude that $\phi$ is a polynomial in
$z^{\alpha}$ and $\bar{z}^{\alpha}$ with complex coefficients,
homogeneous of weight $w$ with respect to the non-standard gradation on
$\mathbb{C}[z^{\alpha}, \bar{z}^{\alpha}]$ in which the weight of
$z^{\alpha}$ and $\bar{z}^{\alpha}$ is $w^{\alpha}$.

To end the proof it amounts to show that $\phi$ does not depend on the
conjugate variables $\bar{z}^{\alpha}$. Under an additional
assumption that $\phi$ is holomorphic, this is immediate. For future
purposes we would like, however, to assume only the smoothness of $\phi
$. The argument will be inductive with respect to the weight $w$. Cases
$w=0$ and $w=1$ are trivial. For a general $w$ let us fix an index
$\alpha$, and denote $\phi=\phi(z^{\alpha},z^{\beta}, \bar
{z}^{\gamma})$ where $\beta\neq\alpha$. Denote by $\phi_{\alpha
}$ the derivative of $\phi$ with respect to $z^{\alpha}$. Using the
homogeneity of $\phi$ we easily get that for every $\xi\in\mathbb{C}$
\begin{align*}
\phi_{\alpha}(\xi^{w^{\alpha}}z^{\alpha}, \xi^{w^{\beta}}
z^{\beta}, \bar{\xi}^{w^{\gamma}}\bar{z}^{\gamma})&=\lim_{|h|\to
0}\frac{\phi(\xi^{w^{\alpha}}z^{\alpha}+h,\xi^{w^{\beta}}
z^{\beta},\bar{\xi}^{w^{\gamma}}\bar{z}^{\gamma})-\phi(\xi
^{w^{\alpha}}z^{\alpha}, \xi^{w^{\beta}} z^{\beta},\bar{\xi
}^{w^{\gamma}}\bar{z}^{\gamma})}{h}
\\
&=\lim_{|h'|\to0}\frac{\phi(\xi^{w^{\alpha}}(z^{\alpha}+h'), \xi
^{w^{\beta}} z^{\beta},\bar{\xi}^{w^{\gamma}}\bar{z}^{\gamma
})-\phi(\xi^{w^{\alpha}}z^{\alpha},\xi^{w^{\beta}} z^{\beta
},\bar{\xi}^{w^{\gamma}}\bar{z}^{\gamma})}{\xi^{w^{\alpha}}\cdot
h'}\\
&=\lim_{|h'|\to0}\frac{\xi^{w}\cdot\phi(z^{\alpha}+h', z^{\beta
},\bar{z}^{\gamma})-\xi^{w}\cdot\phi(z^{\alpha}, z^{\beta},\bar
{z}^{\gamma})}{\xi^{w^{\alpha}}\cdot h'}\\
&=\xi^{w-w^{\alpha}}\phi
_{\alpha}(\xi^{w^{\alpha}}z^{\alpha}, \xi^{w^{\beta}} z^{\beta
},\bar{\xi}^{w^{\gamma}}\bar{z}^{\gamma})\ .
\end{align*}
In other words, $\phi_{\alpha}$ is homogeneous of weight $w-w^{\alpha
}<w$ and thus, by the inductive assumption, a~homogeneous polynomial of
weight $w-w^{\alpha}$ in variables $(z^{\alpha},z^{\beta})$. We conclude that
$\phi=\psi+\eta$, where $\psi$ is a homogeneous polynomial of
weight $w$ in variables $(z^{\alpha},z^{\beta})$ and $\eta$ is a homogeneous
polynomial of weight $w$ in variables $z^{\beta}$ and $\bar
{z}^{\gamma}$ where $\beta\neq\alpha$. Repeating the above
reasoning several times for other indices and the polynomial $\eta$ we
will show that
$\phi=\phi'+\eta'$, where $\phi'$ is a homogeneous polynomial of
weight $w$ in variables $z^{\gamma}$ and $\eta'$ is a homogeneous
polynomial of weight $w$ in variables $\bar{z}^{\gamma}$. In such a
case $\eta'$ should be a complex homogeneous function of weight $w$ as
a difference of two complex homogeneous functions $\phi$ and $\psi'$,
both of weight $w$. However, since $\eta'$ is a homogeneous polynomial
in $\bar{z}^{\gamma}$ we have
\[
\eta'(\bar{\theta}^{w^{\gamma}}\bar{z}^{\gamma})=\bar{\theta
}^{w}\cdot\eta'(\bar{z}^{\gamma})\ ,
\]
hence the only possibility for $\eta'$ to be complex homogeneous of
weight $w$ is $\eta'\equiv0$. This ends the proof. 
\end{proof}

Analogously to the real case, we can define a complex graded bundle as
a smooth fiber bundle with the typical fiber being a complex graded
space. In case that the base possesses a complex manifold structure,
and that the local trivializations are glued by holomorphic functions
(in particular, the total space of the bundle is a complex manifold
itself) we speak about holomorphic graded bundles.

\begin{definition}
\label{def:cplx_gr_bundle}
A \emph{complex graded bundle of rank $\mathbf{d}$} is a smooth fiber
bundle $\tau: E\to M$ over a real smooth manifold $M$ with the typical
fiber $\mathbb{C}^{\mathbf{d}}$ considered as a complex graded space
of rank $\mathbf{d}$. Equivalently, $\tau$ admits local
trivializations $\phi_{U}: \tau^{-1}(U)\to U \times\mathbb
{C}^{\mathbf{d}}$ such that transition functions $g_{UU'}(q):= \phi
_{U'}\circ\phi_{U}^{-1}|_{\{q\}\times\mathbb{C}^{\mathbf{d}}}:
\mathbb{C}^{\mathbf{d}}\to\mathbb{C}^{\mathbf{d}}$ are isomorphism
of complex graded spaces smoothly depending on $q\in U\cap U'$. If $M$
and $E$ are complex manifolds and $g_{UU'}$ are holomorphic functions
of $q$, then $\tau$ is called a \emph{holomorphic graded bundle}.

By a \emph{degree} of a complex (holomorphic) graded bundle we shall
understand the degree of its typical fiber~$\mathbb{C}^{\mathbf{d}}$.

A \emph{morphism of complex graded bundles} is defined as a
fiber-bundle morphism being a complex graded-space morphism on fibers.
A \emph{morphism of holomorphic graded bundles} is a holomorphic map
between the total spaces of the considered holomorphic graded bundles
being simultaneously a morphism of complex graded bundles.
Clearly, complex (holomorphic) graded bundles together with their
morphisms form a \emph{category}.
\end{definition}
%
\begin{remmark}\label{rem:complexVSholomorphic}
Note that a complex graded bundle $\tau: E\rightarrow M$ in which $E$
and $M$ are complex manifolds and the projection $\tau$ is a
holomorphic map needs not to be a holomorphic graded bundle. We also
need to assume that the action $h$ is also holomorphic, i.e. the
complex structure on each fiber is actually the restriction of the
holomorphic structure of $E$. To see this consider a complex rank $1$
vector bundle $C \subset E:=\mathbb{C}^{*}\times\mathbb{C}^{2}$
($\mathbb{C}^{*} = \mathbb{C}\setminus\{0\}$) given
in natural holomorphic coordinates $(x; y^{1}, y^{2})$ on $E$ by the equation
\[
C = \{(x; y^{1}, y^{2}): x y^{1} + \bar{x} y^{2} = 0, x\in\mathbb
{C}^{*} \}.
\]
We shall construct a degree $2$ complex (but not holomorphic) graded
bundle structure on $E$. Set $Y^{1}:= x y^{1} + \bar{x} y^{2}$ and
$Y^{2} := -x y^{1} +\bar{x} y^{2}$. We may take $(x; Y^{1}, Y^{2})$ as
a global coordinate system on $E$ and assign weights $1$, $2$ to
$Y^{1}$, $Y^{2}$, respectively, to define a complex graded bundle
structure on the fibration $\tau: E\rightarrow\mathbb{C}^{*}$. Then
$C$ coincides with the core of $E$, which in every holomorphic graded
bundle should be a complex submanifold of the total space. However, $C$
is not a complex submanifold of $E$, and thus it is impossible to find
a homogeneous holomorphic atlas on $E$. The associated action by
homotheties $h: \mathbb{C}\times E \rightarrow E$ reads as
\[
h(\xi, (x; y^{1}, y^{2})) = (x; \frac{1}{2} (\xi+\xi^{2}) y^{1} +
\frac{\bar{x}}{2x}(\xi-\xi^{2}) y^{2}, \frac{x}{\bar{x}} (\xi
-\xi^{2})y^{1} + \frac{1}{2} (\xi+\xi^{2}) y^{2}).
\]
This action is smooth but not holomorphic, hence it induces a complex
but not holomorphic graded bundle structure according to the forthcoming
\reftext{Theorems~\ref{thm:equiv_holom} and \ref{thm:equiv_complex}}.
\end{remmark}

In what follows, to avoid possible confusions, we will use notation
$\tau:E\rightarrow M$ for complex and smooth graded bundles and $\tau
:F\rightarrow N$ for holomorphic graded bundles.

\begin{example}\label{ex:hol_jet_bundle}
The holomorphic jet bundle $\mathrm{J}^{k}N$ is a canonical example of
a holomorphic (and thus also complex) graded bundle. This fact is
justified analogously to the real case (see Ex. \ref{ex:TkM}).
\end{example}

Finally, we can rewrite \reftext{Definition~\ref{def:hgm_str}} in the complex context.

\begin{definition}\label{def:hgm_str_complex}
A \emph{complex} (respectively, \emph{holomorphic}) \emph
{homogeneity structure} on a smooth (resp., complex) manifold $M$ is a
smooth (resp., holomorphic) action of the multiplicative monoid
$(\mathbb{C},\cdot)$
\[
h:\mathbb{C}\times M \longrightarrow M\ .
\]
A \emph{morphism} of two complex (resp., holomorphic) homogeneity
structures $(M_{1},h^{1})$ and $(M_{2},h^{2})$ is a smooth (resp.,
holomorphic) map $\Phi:M_{1}\rightarrow M_{2}$ intertwining the
actions $h^{1}$ and $h^{2}$, i.e.,
\[
\Phi(h^{1}_{\xi}(p))=h^{2}_{\xi}(\Phi(p)),
\]
for every $p\in M_{1}$ and every $\xi\in\mathbb{C}$. Clearly,
complex (resp., holomorphic) homogeneity structures with their
morphisms form a \emph{category}.
\end{definition}

It is clear, that with every complex (holomorphic) graded bundle $\tau
:E\rightarrow M$ one can associate a natural complex (holomorphic)
homogeneity structure $h^{E}:\mathbb{C}\times E\rightarrow E$ defined
fiber-wise by the canonical $(\mathbb{C},\cdot)$ actions $h^{\mathrm
{V}}$ where $\mathrm{V}=\tau^{-1}(p)$ for every $p\in M$. We shall
call it the action by \emph{homotheties} of $E$. In the remaining part
of this section we shall study the relations between the notions of a
homogeneity structure and a graded bundle in the complex and
holomorphic settings. Our goal is to obtain analogs of \reftext{Theorem~\ref{thm:eqiv_real}} in these two situations.

\paragraph*{A holomorphic action of the monoid of complex numbers}
In the holomorphic setting the results of \reftext{Theorem~\ref{thm:eqiv_real}}
have their direct analog.

\begin{theorem}\label{thm:equiv_holom}
The categories of (connected) holomorphic graded bundles and
(connected) holomorphic homogeneity structures are equivalent. At the
level of objects this equivalence is provided by the following two constructions
\begin{itemize}
\item With every holomorphic graded bundle $\tau:F\rightarrow N$ one
can associate the holomorphic homogeneity structure $(F,h^{F})$, where
$h^{F}$ is the action by the homotheties of $F$.
\item Given a holomorphic homogeneity structure $(N,h)$, the map
$h_{0}:N\rightarrow N_{0}:=h_{0}(N)$ provides $N$ with a canonical
structure of a holomorphic graded bundle such that $h$ is the related action by homotheties.
\end{itemize}
At the level of morphisms: every morphism of holomorphic graded bundles
is a morphism of the related holomorphic homogeneity structures and,
conversely, every morphism of holomorphic homogeneity structures
respects the related canonical holomorphic graded bundle structures.
\end{theorem}

To prove the above theorem we will need two technical results.

\begin{lemma}\label{lem:M0_holom}
Let $N$ be a connected complex manifold and let $\Phi: N\rightarrow N$
be a holomorphic map satisfying $\Phi\circ\Phi=\Phi$. Then the
image $N_{0}:=\Phi(N)$ is a complex submanifold of $N$.
\end{lemma}
\begin{proof}
An analogous result for smooth manifolds is given in Theorem 1.13 in
\cite{Kolar_Michor_Slovak_nat_oper_diff_geom_1993}. Its proof can be
almost directly rewritten in the complex setting. Namely, from the
proof given in \cite{Kolar_Michor_Slovak_nat_oper_diff_geom_1993} we
know that there is an open neighborhood $U$ of $N_{0}$ in $N$ such that
the tangent map $\mathrm{T}_{p} \Phi: \mathrm{T}_{p} N_{\mathbb
{R}}\rightarrow\mathrm{T}_{\Phi(p)} N_{\mathbb{R}}$ has a constant
rank while $p$ varies in $U$. Therefore, due to the identification
$T_{p}N_{\mathbb{R}}\approx\mathrm{T}'_{p} N$, the map $\mathrm
{T}'_{p} \Phi: \mathrm{T}'_{p} N \rightarrow\mathrm{T}'_{\Phi(p)}
N$ also has a constant rank.

Now take any $q\in N_{0}$, so $\Phi(q)=q$. From the constant rank
theorem for complex manifolds (see e.g., \cite{Gauthier_SCV} or \cite
{Kaup_book}) we can find two charts $(O, v)$ and $(\underline{O},
\underline{v})$ on $N$, both centered at $q$, such that $\underline
{v}\circ\Phi\circ v^{-1}$ is a projection of the form $(z^{1}, \ldots
, z^{n})\mapsto(z^{\underline{1}}, \ldots,z^{\underline{n}}, 0,
\ldots, 0)$, where $\underline{n}\leq n$ is the rank of $\mathrm
{T}'_{q} \Phi$. We conclude that $O\cap N_{0}$ is a complex
submanifold of $N$ of dimension $\underline{n}$. Since $q\in N_{0}$
was arbitrary and $N_{0}$ is connected, the assertion follows.

\end{proof}

\begin{lemma} \label{lem:complex_gr_sspace}
Let $\mathrm{V}$ be a complex graded space and $\mathrm{V}'\subset
\mathrm{V}$ a complex submanifold invariant with respect to the action
of the homotheties of $\mathrm{V}$, i.e., $h^{\mathrm{V}}_{\xi
}(\mathrm{V}')\subset\mathrm{V}'$ for any $\xi\in\mathbb{C}$.
Then $\mathrm{V}'$ is a complex graded subspace of $\mathrm{V}$.
\end{lemma}

\begin{proof}
Observe first that $0=h^{\mathrm{V}}_{0}(\mathrm{V}')=h^{\mathrm
{V}}_{0}(\mathrm{V})$ lies in $\mathrm{V}'$. Denote by $(z^{\alpha
}_{w})_{\alpha\in I}$ a system of complex graded coordinates on
$\mathrm{V}$. Since $\mathrm{V}'$ is a complex submanifold, we may
choose a subset $I'\subset I$ of cardinality $\dim_{\mathbb
{C}}\mathrm{V}'$ such that the differentials
%
\begin{equation}\label{eqn:differentials_W}
\mathrm{d}_{q} z^{\alpha}_{w}|_{\mathrm{T}'_{q}\ \mathrm
{V}'},\quad\alpha\in I',
\end{equation}
are linearly independent (over $\mathbb{C}$) at $q=0$. In consequence,
the restrictions $(z^{\alpha}_{w}|_{\mathrm{V}'})_{\alpha\in I'}$
form a coordinate system for $\mathrm{V}'$ around $0$. The idea is to
show that these functions form a global graded coordinate system on
$\mathrm{V}'$.

Note that $\mathrm{V}$ has an important property that it can be
recovered from an arbitrary open neighborhood of $0$ by the action of
$h^{\mathrm{V}}$. Since $\mathrm{V}'\subset\mathrm{V}$ is
$h^{\mathrm{V}}$-invariant it also has this property.
As has been already showed, the differentials \reftext{\eqref{eqn:differentials_W}}
are linearly independent for $q\in U\cap\mathrm{V}'$ where $U$ is a
small neighborhood of $0$ in $\mathrm{V}$. Using the equality
\[
\left\langle\mathrm{d}_{h(\xi,q)} f^i_w, (\mathrm{T}h_\xi)
v_q \right\rangle= \xi^{w} \,
\left\langle\mathrm{d}_q f^i_w, v_q\right\rangle
\]
where $f^{i}_{w}: \mathrm{V}\rightarrow\mathbb{C}$ is a function of
weight $w$ and $v_{q}\in\mathrm{T}_{q}\mathrm{V}$, and the property
that $\mathrm{V}'$ is generated by $h^{\mathrm{V}}$ from any
neighborhood of $0$, we conclude that the differentials \reftext{\eqref{eqn:differentials_W}} are linearly independent for any $q\in\mathrm
{V}'$. Thus $(z^{\alpha}_{w}|_{\mathrm{V}'})_{\alpha\in I'}$ is a global
system of graded coordinates for $\mathrm{V}'$. This ends the proof.

\end{proof}

\begin{corollary} \label{cor:complex_gr_subbundle}
Let $\tau': E'\rightarrow M'$ be a complex graded subbundle of a
holomorphic graded bundle $\tau: F \rightarrow M$ such that $E'$ is a
complex submanifold of $F$. Then $E'$ is a holomorphic graded subbundle.
\end{corollary}
\begin{proof} First of all the base $M':=M\cap E'$ of $E'$ is a complex
submanifold of $M$. To prove this we apply \reftext{Lemma~\ref{lem:M0_holom}}
with $\Phi= \tau|_{E'}: E'\rightarrow E'$.

Being a holomorphic subbundle is a local property: if any point $q\in
M'$ has an open neighborhood $U'\subset M'$ such that $E'|_{U'}$ is a
holomorphic subbundle, then $E'$ is itself a holomorphic subbundle of
$F$. Indeed, we know that transition maps of $E'$ are holomorphic since
$E'$ is a holomorphic submanifold. Moreover, by \reftext{Lemma~\ref{lem:hol_hmg_fun}}, these maps are also polynomial on fibers.

Thus take $q\in M'$ and denote graded fiber coordinates of $\tau:
F\rightarrow M$ by $(z^{\alpha}_{w})_{\alpha\in I}$. They are
holomorphic functions defined on $F|_{U}$ for some open subset
$U\subset M$, $q\in U$. Let $(x^{a})$ be coordinates on $U'\subset M'$
around~$q$. Take a subset $I'\subset I$ such that functions $(x^{a},
z^{\alpha}_{w})_{\alpha\in I'}$ form a coordinate system for $E$
around~$q$.
Then the differentials ($\alpha\in I'$)
\[
\mathrm{d}_{\tilde{q}} z^{\alpha}_{w}|_{\mathrm{T}'_{q}
E'},\quad\mathrm{d}_{\tilde{q}} x^{a}|_{\mathrm{T}'_{q} E'}
\]
are still linearly independent for $\tilde{q}$ in some open
neighborhood $\tilde{U}$ of $q$ in $F$, possibly smaller than $U$. It
follows from the proof of \reftext{Lemma~\ref{lem:complex_gr_sspace}} that
$(z^{\alpha}_{w})_{\alpha\in I'}$ form a global coordinate system on
each of the fibers of $E'$ over $q'\in U'':= \tilde{U}\cap U'$. Thus
$(x^{a}, z^{\alpha}_{w})_{\alpha\in I'}$ is a graded coordinate
system for the subbundle $E'|_{U''}$ consisting of holomorphic
functions turning it into a holomorphic graded bundle. This finishes
the proof.
\end{proof}

Now we are ready to prove \reftext{Theorem~\ref{thm:equiv_holom}}.

\begin{proof}[Proof of \reftext{Theorem~\ref{thm:equiv_holom}}:]
The crucial step is to show, that a holomorphic action $h:\mathbb
{C}\times N\rightarrow N$ of the multiplicative monoid $(\mathbb
{C},\cdot)$ on a connected complex manifold $N$, determines the
structure of a holomorphic graded bundle on $h_{0}: N\rightarrow
N_{0}:=h_{0}(N)$.

Clearly, since $h$ is holomorphic, so is $h_{0}$. What is more, this
map satisfies $h_{0}\circ h_{0} = h_{0}$ and thus, by \reftext{Lemma~\ref{lem:M0_holom}}, $N_{0}=h_{0}(N)$ is a complex submanifold of $N$. Now
notice that the restriction of $h$ to $\mathbb{R}\times N$ gives an action
\[
h^{\mathbb{R}}: \mathbb{R}\times N\rightarrow N
\]
of the monoid $(\mathbb{R}, \cdot)$. In view of \reftext{Theorem~\ref{thm:eqiv_real}}, $h_{0}:N\rightarrow N_{0}$ is a (real) graded bundle
(cf. the proof of \reftext{Lemma~\ref{lem:hol_hmg_fun}}), say, of degree $k$.

We shall now follow the ideas from the proof of \reftext{Theorem~\ref{thm:eqiv_real}} provided in \cite{JG_MR_gr_bund_hgm_str_2011}. The
crucial step is to embed $N$ into the holomorphic jet bundle $\mathrm
{J}^{k} N$ as a holomorphic graded subbundle. Recall (cf. Ex. \ref{ex:hol_jet_bundle}) that $\mathrm{J}^{k}N$ has a canonical
holomorphic graded bundle structure.
Consider a map
\[
\phi^{\mathbb{C}}:N\rightarrow\mathrm{J}^{k} N, \quad q\mapsto
\mathbf{j}^{k}_{\xi=0} h(\xi, q),\quad\xi\in\mathbb{C},
\]
sending each point $q\in N$ to the $k$th-holomorphic jet
at $\xi=0$ of the holomorphic curve $\mathbb{C}\ni\xi\mapsto h(\xi
,q)\in N$. Clearly, $\phi^{\mathbb{C}}$ is a holomorphic map as a
lift of a holomorphic curve to the holomorphic jet bundle $\mathrm
{J}^{k} N$. The composition of $\phi^{\mathbb{C}}$ with the canonical
isomorphism $\mathrm{J}^{k} N \simeq\mathrm{T}^{k} N_{\mathbb{R}}$
gives the map
\[
\phi^{\mathbb{R}}:N_{\mathbb{R}}\rightarrow\mathrm{T}^{k}
N_{\mathbb{R}}, \quad q\mapsto\mathbf{t}^{k}_{t=0} h(t, q), \quad
t\in\mathbb{R}.
\]
As indicated in the proof of Theorem 4.1 of \cite
{JG_MR_gr_bund_hgm_str_2011}, $\phi^{\mathbb{R}}$ is a topological
embedding naturally related with the real homogeneity structure
$h^{\mathbb{R}}$. Thus also $\phi^{\mathbb{C}}$ is a topological
embedding. Therefore, since $\phi^{\mathbb{C}}$ is also holomorphic,
the image $\widetilde{N}:=\phi^{\mathbb{C}}(N)\subset\mathrm
{J}^{k} N$ is a complex submanifold, biholomorphic with $N$. Let us
denote by $\widetilde{h}: \mathbb{C}\times\widetilde{N}\rightarrow
\widetilde{N}$ the corresponding action on $\widetilde{N}$ induced
from $h$ by means of $\phi^{\mathbb{C}}$. Since $\phi^{\mathbb{C}}$
intertwines the action $h$ and the canonical action by homotheties
\[
h^{\mathrm{J}^{k} N}: \mathbb{C}\times\mathrm{J}^{k} N \rightarrow
\mathrm{J}^{k} N
\]
on the bundle of holomorphic $k$th-jets, the action
$\widetilde{h}$ coincides with the restriction of
$h^{\mathrm{J}^{k} N}$ to $\mathbb{C}\times\widetilde{N}$. Hence,
$\widetilde{N}$ is a complex submanifold of $\mathrm{J}^{k} N$
invariant with respect to the action $h^{\mathrm{J}^{k} N}$.
Using \reftext{Lemma~\ref{lem:complex_gr_sspace}} on each fiber of $\mathrm
{J}^{k} N\big|_{N_{0}}\rightarrow N_{0}$ we conclude that $\widetilde
{N}$ is a complex graded subbundle of $\mathrm{J}^{k} N$.
Since $\widetilde{N}\approx N$ was a complex submanifold of $\mathrm
{J}^{k} N$, it is also a holomorphic graded subbundle of $\mathrm
{J}^{k}N$ due to \reftext{Corollary~\ref{cor:complex_gr_subbundle}}.
Thus we have constructed a canonical holomorphic graded bundle
structure on $N$ starting from a holomorphic homogeneity structure
$(N,h)$. Clearly the action by homotheties $h^{N}$ related with this
graded bundle coincides with the initial action $h$.

The above construction, and the construction of a canonical holomorphic
homogeneity structure $(F,h^{F})$ from a holomorphic graded bundle
$\tau: F\rightarrow N$ are mutually inverse, providing the desired
equivalence of categories at the level of objects. \medskip

To show the equivalence at the level of morphisms consider two
holomorphic graded bundles $\tau_{j}:F_{j}\rightarrow N_{j}$, with
$j=1,2$, and let $h^{j}$, with $j=1,2$, be the related homogeneity
structures. Let $\Phi: F_{1} \rightarrow F_{2}$ be a holomorphic map
such that $\Phi\circ h^{1} = h^{2}\circ\Phi$. It is enough to show
that $\Phi$ is a morphism of holomorphic graded bundles. Since $\Phi$
is holomorphic by assumption it suffices to show that on each fiber
$\Phi: (F_{1})_{p} \rightarrow(F_{2})_{\Phi(p)}$ is a morphism of
complex graded spaces.

Let now $(z^{\alpha}_{w})$ and $(\underline{z}^{\alpha}_{w})$ be
graded coordinates on $(F_{1})_{p}$ and $(F_{2})_{\Phi(p)}$, respectively.
Note that $\Phi^{*}\underline{z}^{\alpha}_{w} = \underline
{z}^{\alpha}_{w} \circ\Phi$ is a $\mathbb{C}$-homogeneous function
on $(F_{1})_{p}$, hence in light of \reftext{Lemma~\ref{lem:hol_hmg_fun}} it is
a homogeneous polynomial in $z^{\alpha}_{w}$. Thus, indeed, $\Phi$
has a desired form.
\end{proof}

\paragraph*{On smooth actions of the monoid $(\mathbb{C}, \cdot)$ on
smooth manifolds}
Our goal in the last paragraph of this section is to study smooth
actions of the monoid $(\mathbb{C},\cdot)$ on smooth manifolds, i.e.,
complex homogeneity structures (see \reftext{Definition~\ref{def:hgm_str_complex}}). Contrary to the holomorphic case, there is no
equivalence between such structures and complex graded bundles. To
guarantee such an equivalence we will need to make additional
assumptions. Informally speaking, the real and the imaginary parts of
the action of $(\mathbb{C},\cdot)$ should be compatible. The
following examples should help to get the right intuitions.

\begin{example}
Consider $M=\mathbb{R}$ and define an action $h:\mathbb{C}\times
M\rightarrow M$ by $h(\xi, y)=(|\xi|^{2}\,y)$, where $y$ is a
standard coordinate on $\mathbb{R}$. Clearly this is a smooth action
of the multiplicative monoid $(\mathbb{C},\cdot)$,
but the fibers $M$ admit no structure of a complex graded bundle. This
is clear from dimensional reasons. Indeed, the base $h_{0}(M)$ is just
a single point $0\in\mathbb{R}$ and thus $M$, as a single fiber,
should admit a structure of a complex graded space. This is impossible
as $M$ is odd-dimensional.
\end{example}

\begin{example} Consider $M=\mathbb{C}$ with a standard coordinate
$z:M\rightarrow\mathbb{C}$ and define a smooth multiplicative action
$h:\mathbb{C}\times M\rightarrow M$ by the formula $h(\xi, z)=|\xi
|^{2} \xi\,z$. We claim that $h$ is not a homothety action related
with any complex graded bundle structure on $M$.

Assume the contrary. The basis of $M$ should be $h_{0}(M)=\{0\}$, i.e.
a single point. Thus $M$ is a complex graded space (a complex graded
bundle over a single point), say, of rank $\mathbf{d}$. Clearly the
restriction $h|_{\mathbb{R}}:\mathbb{R}\times M\rightarrow M$ should
provide $M$ with a structure of a (real) graded space of rank $2\mathbf
{d}$. Observe that $h|_{\mathbb{R}}$ is in fact a homothety action on
$\mathbb{R}^{(0,0,2)}$ and thus we should have $M=\mathbb
{C}^{(0,0,1)}$. In such a case, there should exist a global complex
coordinate $\tilde{z}: M \xrightarrow{\simeq} \mathbb{C}$ which is
homogeneous of degree $3$, and so $h_{\varepsilon_{3}}$, where
$\varepsilon_{3}:=e^{2\pi\sqrt{-1}/3}$ is the third order primitive
root of 1, should be the identity on $M$, as there are no coordinates
on $M$ of other weights. Yet, $h(\varepsilon_{3}, z) = \varepsilon
_{3} z \neq z$, thus a contradiction.
\end{example}

The above examples reveal two important facts concerning a complex
homogeneity structure $h:\mathbb{C}\times M\rightarrow M$. First of
all, the restriction of $h$ to $(\mathbb{R},\cdot)$ makes $M$ a
(real) homogeneity structure. Secondly, the action of the primitive
roots of $1$ on $h|_{\mathbb{R}}$-homogeneous functions allows to
distinguish complex graded bundles among all complex homogeneity structures.

\begin{remmark}\label{rem:induced_action_on_tower}
Let $h:\mathbb{C}\times E^k\rightarrow E^k$ be a complex homogeneity structure such that the restriction $h|_{\mathbb{R}}$ makes $\tau= h_{0}: E^{k} \rightarrow M_{0}$ a (real) graded
bundle of degree $k$. For any $\xi
\in\mathbb{C}$ the action $h_{\xi}$ commutes with the homotheties
$h(t, \cdot)$, $t\in\mathbb{R}$, hence $h_{\xi}: E^{k}\rightarrow
E^{k}$ is a (real) graded bundle morphism, in view of \reftext{Theorem~\ref{thm:eqiv_real}}. Therefore $h$ induces an action of the monoid
$(\mathbb{C}, \cdot)$ on each (real) graded bundle $\tau^{j}:
E^{j}\rightarrow M_{0}$ in the tower \reftext{\eqref{eqn:tower_grd_spaces}}, and
on each core bundle $\hat{E}^{j}$, for $j=1,2,\ldots, k$.
\end{remmark}

These observations motivate the following

\begin{definition}\label{def:nice_hgm_str}
Let $h:\mathbb{C}\times E^{k}\rightarrow E^{k}$ be a complex
homogeneity structure such that $\tau= h_{0}: E^{k} \rightarrow M$ is the
(real) graded bundle of degree $k$ associated with $h|_{\mathbb{R}}$.
Denote by $\varepsilon_{2j}:=e^{2\pi\sqrt{-1}/(2j)}$ the
$2j$th-order primitive root of 1, and by
$J_{2j}:=h(\varepsilon_{2j},\cdot):E^{k}\rightarrow E^{k}$ the action
of $\varepsilon_{2j}$ on $E^{k}$. We say that the complex homogeneity
structure $h$ is \emph{nice} if $J_{2j}$ acts as minus identity on the
core bundle $\hat{E^{j}}$ for every $j=1, 2, \ldots, k$.

Nice\ complex homogeneity structures form a \emph{full subcategory} of
the category of complex homogeneity structures.
\end{definition}

\begin{example}\label{ex:nice_hgm_str}
It is easy to see, using local coordinates, that if $\tau:E\rightarrow
M$ is a complex graded bundle, and $h^{E}:\mathbb{C}\times
E\rightarrow E$ the related complex homogeneity structure, then $h^{E}$
is nice.
\end{example}

It turns out that the converse is also true, that is, for nice\ complex
homogeneity structures we can prove an analog of \reftext{Theorem~\ref{thm:equiv_holom}}:

\begin{theorem}\label{thm:equiv_complex}
The categories of (connected) complex graded bundles and
(connected) nice complex homogeneity structures are equivalent. At the level
of objects this equivalence is provided by the following two constructions
\begin{itemize}
\item With every complex graded bundle $\tau:E\rightarrow M$ one can
associate a nice\ complex homogeneity structure $(E,h^{E})$, where
$h^{E}$ is the action by homotheties of $E$.
\item Given a nice\ complex homogeneity structure $(M,h)$, the map
$h_{0}:M\rightarrow M_{0}:=h_{0}(M)$ provides $M$ with a canonical
structure of a complex graded bundle such that $h$ is the related action by homotheties.
\end{itemize}
At the level of morphisms: every morphism of complex graded bundles is
a morphism of the related nice\ complex homogeneity structures and,
conversely, every morphism of nice\ complex homogeneity structures
respects the canonical complex graded bundle structures.
\end{theorem}

Again in the proof we shall need a few technical results. First observe
that $J_{2} = h_{-1}$ acts as minus identity on every vector bundle, so
for $k=1$
the condition in \reftext{Definition~\ref{def:nice_hgm_str}} is trivially
satisfied (i.e., a~degree-one complex homogeneity structure is always
nice). Thus every $(\mathbb{C},\cdot)$-action whose restriction to
$(\mathbb{R},\cdot)$ is of degree one should be a complex bundle.

\begin{lemma}\label{lem:C_action_on_vect_space}
Let $h: \mathbb{C}\times W\rightarrow W$
be a smooth action of the monoid $(\mathbb{C}, \cdot)$ on a real
vector space $W$, such that
$h(t, v) = t\, v$ for every $t\in\mathbb{R}$ and $v\in W$.

Then $h$ induces a complex structure on $W$ by the formula
%
\begin{equation}
h(a+b\sqrt{-1}, v) = a\, v + b\,h(\sqrt{-1}, v)
\end{equation}
for every $a, b\in\mathbb{R}$.
\end{lemma}
\begin{proof}
Denote $J:=h(\sqrt{-1}, \cdot):W\rightarrow W$. We have $J\circ J =
h_{\sqrt{-1}}\circ h_{\sqrt{-1}} = h_{-1}=-\operatorname{id}_{W}$.
Moreover, $J$~is $\mathbb{R}$-linear, since $J$ commutes with the
homotheties $h_{t}: W\rightarrow W$ for any $t\in\mathbb{R}$ (see
Theorem~2.4 \cite{JG_MR_higher_vec_bndls_and_multi_gr_sym_mnflds}).
Therefore, $J$ defines a complex structure on $W$ and the formula
\[
\xi\, v : =\Re\xi\, v + \Im\xi\,J(v),
\]
where $\xi\in\mathbb{C}$ and $v\in W$, allows us to consider $W$ as
a complex vector space.
Note that $\xi\, v = h(\xi, v)$ for $\xi\in\mathbb{R}$. To prove
that this equality holds for any $\xi\in\mathbb{C}$ we shall study the restriction
\[
h_{|S^{1}\times W}: S^{1}\times W\rightarrow W
\]
which is a group action of the unit circle on a complex vector space
$W$. Indeed, for any $\xi\in\mathbb{C}$, the action
$h(\xi, \cdot): W\rightarrow W$ is a $\mathbb{C}$-linear map, since
it commutes with the complex structure $J$ and the endomorphisms $h(t,
\cdot)$ for $t\in\mathbb{R}$. It follows from the general theory
that $W$ splits into sub-representations $W= \bigoplus_{j=1}^{n}
W_{j}$ such that for any $|\theta|=1$ and any $v\in W_{j}$
\[
h(\theta, v) =\theta^{k_{j}} \, v
\]
where $k_{1}, \ldots, k_{n}$ are some integers. The restriction of $h$
to each summand $W_{j}$ defines an action of the monoid $(\mathbb
{C},\cdot)$ hence, without loss of generality, we may assume that
$W=W_{1}$ and that
%
\begin{equation}\label{eqn:h_theta_k}
h(t\,\theta, v) = t\,\theta^{k}\, v
\end{equation}
for every $\theta\in S^{1}$, $t\in\mathbb{R}$ and $v\in W$. Note
that $k$ should be an odd integer as $J^{2}=-\operatorname{id}_{W}$.
Equivalently, taking $\xi=t\, \theta$, we can denote $h(\xi, v) =
\xi^{k}\bar{\xi}^{-k+1}\, v$ for every $\xi\in\mathbb{C}\setminus
\{0\}$. However, the function $\xi\mapsto\xi^{k}\bar{\xi}^{-k+1}$,
$0\mapsto0$, is not differentiable at $\xi=0$ unless $k=1$.
Therefore, $h(\xi, v)=\xi\, v$, as was claimed.
\end{proof}

Now we shall show that an analogous result holds for nice\ homogeneity
structures of arbitrary degree.

\begin{lemma}\label{lem:action_Ek}
Let $h:\mathbb{C}\times M\rightarrow M$ be a nice\ complex homogeneity
structure such that the restriction $h|_{\mathbb{R}}$ makes $M$ a
(real) graded space of degree $k$. Then $M$ is a complex graded space
of degree $k$ with homotheties given by $h$.
\end{lemma}

\begin{proof}
Denote by $\mathrm{W}^{k}$ the (real) graded space structure on $M$.
By $\mathrm{W}^{j}$ with $j\leq k$ denote the lower levels of the
tower \reftext{\eqref{eqn:tower_grd_spaces}} associated with $\mathrm{W}^{k}$.
We shall proceed by induction on $k$. Case $k=1$ follows immediately
from \reftext{Lemma~\ref{lem:C_action_on_vect_space}}.

Let now $k$ be arbitrary. The basic idea of the proof is to define a
complex graded space structure of degree $k-1$ on $\mathrm{W}^{k-1}$
using the inductive assumption and to construct a complex graded space
structure of rank $(0,\ldots,0,\dim_{\mathbb{C}}\widehat{\mathrm
{W}}^{k})$ on the core $\widehat{\mathrm{W}}^{k}$. Then using both
structures we build complex graded coordinates on $\mathrm{W}^{k}$.

Recall (see \reftext{Remark~\ref{rem:induced_action_on_tower}}) that, for any $\xi
\in\mathbb{C}$, the map $h_{\xi}: \mathrm{W}^{k}\rightarrow\mathrm
{W}^{k}$ is a (real) graded space morphism, and that $h$ induces an action of
the monoid $(\mathbb{C}, \cdot)$ on each graded space $\mathrm
{W}^{j}$ in the tower \reftext{\eqref{eqn:tower_grd_spaces}}. Note that the
induced action on $\mathrm{W}^{k-1}$ satisfies all assumptions of our
lemma hence, by the inductive assumption, we may consider $\mathrm
{W}^{k-1}$ as a complex graded space.

Denote by $(z^{\alpha}_{w})$ complex graded coordinates on $\mathrm
{W}^{k-1}$ and pullback them to $\mathrm{W}^{k}$ by means of the
projection $\tau^{k}_{k-1}: \mathrm{W}^{k}\rightarrow\mathrm
{W}^{k-1}$. Denote the resulting functions again with the same symbols.
This should not lead to any confusion since the pullbacked function
$z^{\alpha}_{w}: \mathrm{W}^{k}\rightarrow\mathbb{C}$ is still of
weight $w$, i.e.
%
\begin{equation}\label{eqn:theta_on_Ek-1}
h_{\xi}^{\ast}(z^{\alpha}_{w}) = \xi^{w}\, z^{\alpha}_{w}.
\end{equation}

On the other hand, as any morphism of graded spaces, $h_{\xi}$ can be
restricted to the core $\widehat{\mathrm{W}}^{k}$.
This defines an action of $(\mathbb{C}, \cdot)$ on $\widehat{\mathrm
{W}}^{k}$,
\[
\widehat{h}:= h_{|\mathbb{C}\times\widehat{\mathrm{W}}^{k}}:
\mathbb{C}\times\widehat{\mathrm{W}}^{k}\rightarrow\widehat
{\mathrm{W}}^{k}.
\]
Recall that $\widehat{\mathrm{W}}^{k}$ is a real vector space with
homotheties defined by
%
\begin{equation}\label{eqn:homotheties_in_core_Ek}
(t, v)\mapsto t\ast v:= h|_{\mathbb{R}}(\sqrt[k]{t}, v)\ ,
\end{equation}
for every $t\geq0$ and every $v\in\widehat{\mathrm{W}}^{k}$.
In a (real) graded coordinate system $(\Re z^{\alpha}_{w}, \Im
z^{\alpha}_{w}, p^{\mu})$ on $\mathrm{W}^{k}$, where $p^{\mu}:
\mathrm{W}^{k}\rightarrow\mathbb{R}$ are arbitrary coordinates of
weight $k$, the map \reftext{\eqref{eqn:homotheties_in_core_Ek}} reads $(t,
(\widehat{p}^{\mu}))\mapsto(t\,\widehat{p}^{\mu})$, where
$\widehat{p}^{\mu}=p^{\mu}|_{\widehat{\mathrm{W}}^{k}}$. Hence it
is a smooth map (with respect to the inherited submanifold structure)
which can be extended also to the negative values of $t$.

Let us denote by $\widehat{J}_{4k}$ the restriction of
$J_{4k}:=h(\varepsilon_{4k},\cdot)$ to $\widehat{\mathrm{W}}^{k}$.
By assumption ($h$ is nice), $-\operatorname{id}_{\widehat{\mathrm
{W}}^{k}} = \widehat{J}_{2k} = \widehat{J}_{4k}\circ\widehat
{J}_{4k}$. Therefore (cf. \reftext{Lemma~\ref{lem:C_action_on_vect_space}}),
$\widehat{J}_{4k}$ defines a complex structure on the real vector
space $\widehat{\mathrm{W}}^{k}$ by the formula
\[
(a+b\sqrt{-1})\ast v:= a\ast v + b\ast\widehat{J}_{4k}(v)\ ,
\]
for every $a,b\in\mathbb{R}$ and $v\in\widehat{\mathrm{W}}^{k}$.

Note that the homotheties $h_{\xi}$ commute also with $\widehat
{J}_{4k}$, therefore $h_{\xi}|_{\widehat{\mathrm{W}}^{k}}$ is a
$\mathbb{C}$-linear endomorphism of $\widehat{\mathrm{W}}^{k}$. By
restricting to $|\xi|=1$ we obtain a
representation of the unit circle group $S^{1}$ in $\operatorname
{GL}_{\mathbb{C}}(\widehat{\mathrm{W}}^{k})$. As in the proof of
\reftext{Lemma~\ref{lem:C_action_on_vect_space}}, without loss of generality we
may assume that there exists an integer $m$ such that
\[
\widehat{h}(\theta, v)= \theta^{m}\ast v
\]
for every $|\theta|=1$. Taking $\theta=\varepsilon_{2k}$ we see that
$\varepsilon_{2k}^{m}=-1$, hence $m\equiv k \mod2k$. It follows that
$\widehat{h}(t\, \theta, v) = t^{k}\,\theta^{m} \ast v$, for every
$t\in\mathbb{R}$ and $\theta\in S^{1}$.
However, the function $\xi:=t\,\theta\mapsto t^{k}\theta^{m} = \xi
^{k} \, (\xi/\bar{\xi})^{m-k}$ is smooth only if $m=k$, since $m-k$
is a multiplicity of $2k$. Therefore,
%
\begin{equation}\label{eqn:theta_on_core}
h(\xi, v) = \xi^{k}\ast v,
\end{equation}
for every $\xi\in\mathbb{C}$ and every $v\in\widehat{\mathrm
{W}}^{k}$. In other words, $\widehat{W}^{k}$ is a complex graded space
of rank $\mathbf{d}=(0,0,\ldots,0,\dim_{\mathbb{C}}\widehat
{\mathrm{W}}^{k})$.\smallskip

Let $(\widehat{z}^{\mu}: \widehat{\mathrm{W}}^{k}\rightarrow
\mathbb{C})$ be a system of complex graded coordinates on $\widehat
{\mathrm{W}}^{k}$. We shall show that it is possible to extend each
$\widehat{z}^{\mu}$ to a complex function $z^{\mu}_{k}: W^{k}
\rightarrow\mathbb{C}$ in such a way that
%
\begin{equation}\label{eqn:theta_on_Ek}
z^{\mu}_{k}(h(\xi, v)) = \xi^{k}\, z^{\mu}_{k}(v)
\end{equation}
hold for any $\xi\in\mathbb{C}$ and $v\in\mathrm{W}^{k}$, i.e.,
$z^{\mu}_{k}$ are complex homogeneous function of weight $k$. First see
that we can find extensions of $\widehat{z}^{\mu}$ which satisfy \reftext{\eqref{eqn:theta_on_Ek}} for $\xi\in\mathbb{R}$. Indeed, the
restriction to $\widehat{W}^{k}$ of a real homogeneous weight $k$
function $W^{k}\rightarrow\mathbb{R}$ can be an arbitrary linear
function on $\widehat{W}^{k}$. Thus we extend the real and imaginary
parts of $\widehat{z}^{\mu}$ separately and get $\mathbb
{R}$-homogeneous extensions, say $\tilde{z}^{\mu}: W^{k} \rightarrow
\mathbb{C}$. Now consider a function
%
\begin{equation}\label{eqn:extensions_of_z_mu}
z^{\mu}_{k}(v) := \frac{1}{2\pi} \int_{|\xi|=1} \xi^{-k} \tilde
{z}(h_{\xi}(v)) \mathrm{d}\lambda(\xi),
\end{equation}
where $v\in W^{k}$ and $\lambda$ is a homogeneous measure on the circle
$S^{1}=\{|\xi|=1\}$ with $\lambda(S^{1})=2\pi$.
Clearly, $z^{\mu}_{k}: W^{k} \rightarrow\mathbb{C}$ is smooth and
(as $h_{\xi}(h_{\theta}(v)) = h_{\xi\theta}(v)$ for any $\theta,\xi\in
\mathbb{C}$) we have
\[
z^{\mu}_{k}(h_{\theta}(v)) = \theta^{k} \frac{1}{2\pi} \int_{|\xi
|=1} \theta^{-k} \xi^{-k} \tilde{z}^{\mu}(h_{\xi\theta}(v))
\mathrm{d}\lambda(\xi) = \theta^{k} z^{\mu}_{k}(v),
\]
for any $|\theta|=1$. Since $\tilde{z}^{\mu}$ is $\mathbb
{R}$-homogeneous, the same is $z^{\mu}_{k}$,
hence $z^{\mu}_{k}$ is actually a complex homogeneous function, as
$S^{1}$ and $\mathbb{R}$ generate $\mathbb{C}$ as a monoid. Moreover,
for $v\in\widehat{W}^{k}$ equality $\tilde{z}^{\mu}(h_{\xi}(v)) = \xi
^{k} \tilde{z}^{\mu}(v)$ holds, hence $z^{\mu}_{k}|_{\widehat{W}^{k}} =
\tilde{z}^{\mu}|_{\widehat{W}^{k}}$, and so $z^{\mu}_{k}$ is indeed
an extension of $\hat{z}^{\mu}$.

Lastly, the system of homogeneous functions $(z^{\alpha}_{w}, z^{\mu
}_{k})$ defines a global diffeomorphism $W^{k}\xrightarrow{\simeq}
\mathbb{C}^{|\mathbf{d}|}$ where $\mathbf{d}=(d_{1}, \ldots,
d_{k})$, $d_{j}=\operatorname{dim}_{\mathbb{C}} \widehat{W}^{j}$.
Indeed, it is enough to point that for any $1\leq j\leq k$, the
restrictions of some of these functions (those of weight $j$) to the
core $\widehat{W}^{j}$ define a diffeomorphism $\widehat
{W}^{j}\xrightarrow{\simeq} \mathbb{C}^{d_{j}}$.
\end{proof}

\reftext{Theorem~\ref{thm:equiv_complex}} is a simple consequence of the above result.

\begin{proof}[Proof of \reftext{Theorem~\ref{thm:equiv_complex}}]
We already observed (see Ex. \ref{ex:nice_hgm_str}) that a complex
graded bundle structure on $\tau:E\rightarrow M$ induces a nice\
complex homogeneity structure $(E,h^{E})$ by the associated action of the
homotheties of $E$.

The converse is also true. Indeed, let $h:\mathbb{C}\times
M\rightarrow M$ be a nice\ complex homogeneity structure. Clearly
$h|_{\mathbb{R}}$ is a (real) homogeneity structure on $M$, and thus,
by \reftext{Theorem~\ref{thm:eqiv_real}}, $h_{0}:M\rightarrow M_{0}:=h_{0}(M)$
is a (real) graded bundle of degree, say, $k$. Now on each fiber of
this bundle the action $h$ defines a nice\ homogeneity structure. By
applying \reftext{Lemma~\ref{lem:action_Ek}} we get a complex graded space
structure on each fiber of $h_{0}$. Thus $M$ is indeed a complex graded
bundle.

The equivalence at the level of morphisms is showed analogously to the
holomorphic case: it amounts to show that a smooth map between two
complex graded spaces which intertwines the homothety actions is a
complex graded space morphism. This is precisely the assertion of \reftext{Lemma~\ref{lem:hol_hmg_fun}}.
\end{proof}

\section{Actions of the monoid $\mathcal{G}_{2}$}
\label{sec:g2}
In this section we shall study smooth actions of the monoid $\mathcal
{G}_{2}$ on smooth manifolds.

\paragraph*{The left and right actions of the monoid $\mathcal{G}_{2}$}
Recall (see the end of Section~\ref{sec:perm}) that $\mathcal{G}_{2}$
was introduced as the space of 2nd-jets of punctured
maps $\phi:(\mathbb{R},0)\rightarrow(\mathbb{R},0)$ with the
multiplication induced by the composition. Under the identification of
$\mathcal{G}_{2}$ with $\mathbb{R}^{2}=\{(a,b)\ |\ a,b\in\mathbb
{R}\}$ it reads as (see \reftext{Remark~\ref{rem:g_k_k=1_2}}):
%
\begin{equation}
\label{eqn:G_2_multiplication}
(a,b)(A,B)=(aA,a B+bA^{2})\ .
\end{equation}
Since this multiplication is clearly non-commutative, unlike in the
case of real or complex numbers, we have to distinguish between left
and right actions of $\mathcal{G}_{2}$.

The crucial observation about $\mathcal{G}_{2}$ is that it contains
two submonoids:
\begin{itemize}
\item the multiplicative reals $(\mathbb{R}, \cdot)\simeq\{(a, 0):
a\in\mathbb{R}\}$, corresponding to the 2nd-jets of
punctured maps $\phi(t)=at$ for $a\in\mathbb{R}$,
\item and the additive group $(\mathbb{R}, +) \simeq\{(1, b): b\in
\mathbb{R}\}$.
\end{itemize}
Now to study right (or left) smooth actions of $\mathcal{G}_{2}$ on a
smooth manifold $M$ we use a technique similar to the one used to study
$(\mathbb{C},\cdot)$-actions in Section~\ref{sec:complex}. We begin
by considering the action of $(\mathbb{R},\cdot)\subset\mathcal
{G}_{2}$ which, by \reftext{Theorem~\ref{thm:eqiv_real}}, makes $M$ a (real)
graded bundle. Actually it will be more convenient to speak of the
related weight vector field (see \reftext{Definition~\ref{def:euler_vf}}) in
this case. On the other hand, the action of the additive reals
$(\mathbb{R},+)$ is a flow, i.e. it is encoded by a single (complete)
vector field on $M$. It is now crucial to understand the relation
(compatibility conditions) between these two structures. This can be
done by looking at the formula
%
\begin{equation}
\label{eqn:commutation_submonoids}
(a,0)(1,b/a)=(a,b)=(1,b/a^{2})(a,0)\ ,
\end{equation}
which allows to decompose every element of ${\mathcal{G}}^{\text
{inv}}_{2}=\mathcal{G}_{2} \setminus\{(0, b): b\neq0\}$, the group
of invertible elements of $\mathcal{G}_{2}$, as a product of the elements of the
submonoids $(\mathbb{R},\cdot)$ and $(\mathbb{R},+)$. Since equation \reftext{\eqref{eqn:commutation_submonoids}} describes the commutation of the
two submonoids, it helps to express the compatibility conditions of the
two related structures at the infinitesimal level as the following
result shows. Recall the notion of a homogeneous vector field -- cf.
\reftext{Definition~\ref{def:hgm_real}} and \reftext{Remark~\ref{rem:euler_weight}}.

\begin{lemma}\label{lem:G2_actions_infinitesimally}
Every smooth right (respectively, left) action $H: M \times\mathcal
{G}_{2} \rightarrow M$ (resp., $H: \mathcal{G}_{2}\times M \rightarrow
M$) on a smooth manifold $M$ provides $M$ with:
\begin{itemize}
\item a canonical graded bundle structure $\pi: M\rightarrow
M_{0}:=H_{(0,0)}(M)$ induced by the action of the submonoid $(\mathbb
{R},\cdot)\subset\mathcal{G}_{2}$,
\item and a complete vector field $X\in\mathfrak{X}(M)$ of weight $-1$
(resp., $Y\in\mathfrak{X}(M)$ of weight $+1$) with respect to the
above graded structure on $M$.
\end{itemize}
In other words, any $\mathcal{G}_{2}$-action provides $M$ with two
complete vector fields: the weight vector field $\Delta$ and another
vector field $X$ (respectively, $Y$), such that their Lie bracket satisfies
%
\begin{equation}
\label{eqn:commutator_X_Delta}
[\Delta,X]=-X \quad(\text{resp., $[\Delta, Y]=Y$})\ .
\end{equation}
\end{lemma}

\begin{proof}
By \reftext{Theorem~\ref{thm:eqiv_real}}, the homogeneity structure $h: \mathbb
{R}\times M\rightarrow M$ obtained as the restriction of $H$ to the
submonoid $(\mathbb{R},\cdot)\subset\mathcal{G}_{2}$ (i.e.,
$h_{a}=H_{(a,0)}$) defines a graded bundle structure on
$h_{0}:M\rightarrow M_{0}$. Clearly, the flow of the corresponding
weight vector field $\Delta$ is given by $t\mapsto H_{(e^{t},0)}$.

Consider first the case when $H$ is a right $\mathcal{G}_{2}$-action.
Let $X\in\mathfrak{X}(M)$ be the infinitesimal generator of the
action of $s\mapsto H_{(1,s)}$.
In order to calculate the Lie bracket of $\Delta$ and $X$ we will
calculate the corresponding commutator of flows, i.e.
\begin{align*}
&X^{-s}\circ\Delta^{-t} \circ X^{s}\circ\Delta^{t} =
H_{(1,-s)}\circ H_{(e^{-t},0)}\circ H_{(1,s)}\circ H_{(e^{t},0)}\\
&\quad =H_{(e^{t},0)(1,s)(e^{-t},0)(1,-s)}
\overset{\text{\reftext{\eqref{eqn:G_2_multiplication}}}}{=}
H_{(1,-ts+o(ts))}= X^{-ts+o(ts)}\ .
\end{align*}
The latter should correspond to the $ts$-flow of $[\Delta, X]$ and
hence \reftext{\eqref{eqn:commutator_X_Delta}} holds.

In the case when $H$ is a left $\mathcal{G}_{2}$-action denote by $Y$
the infinitesimal generator of the action $s\mapsto H_{(1,s)}$. Now
$H_{g}\circ H_{g'} = H_{gg'}$, so the commutator $Y^{-s}\Delta^{-t}
\circ Y^{s}\circ\Delta^{t}$ equals
$H_{(1,-s)(e^{-t},0)(1,s)(e^{t},0)} = Y^{ts + o(ts)}$, hence $[\Delta,
Y]= Y$.
\end{proof}

\begin{example}\label{ex:T2M}
Consider the right $\mathcal{G}_{2}$-action on $\mathrm{T}^{2} M$
described in \reftext{Example~\ref{ex:TkM_Gk}}. We shall use standard
coordinates $(x^{i}, \dot{x}^{i}, \ddot{x}^{i})$ on $\mathrm{T}^{2}
M$.\footnote{If $[\gamma]_{2} \sim(x^{i}, \dot{x}^{i}, \ddot
{x}^{i})$, then $\gamma(t) = (x^{i} + t \dot{x}^{i} + \frac{1}{2}
\ddot{x}^{i} + o(t^{2}))$.} We have
\[
(x^{i},\dot{x}^{i},\ddot{x}^{i}).(a,b)=(x^{i},a\dot
{x}^{i},a^{2}\ddot{x}^{i}+b\dot{x}^{i}).
\]
It is clear that in this case the homogeneity structure on $\mathrm
{T}^{2} M$ is just the standard degree 2 homogeneity structure, the
weight vector field equals $\Delta= \dot{x}^{i}\partial_{\dot
{x}^{i}} + 2\, \ddot{x}^{i} \partial_{\ddot{x}^{i}}$, and that the
additional vector field $X$ of weight $-1$ is simply $X=\dot
{x}^{i}\partial_{\ddot{x}^{i}}$.
\end{example}

\begin{example}
Let us now focus on the left $\mathcal{G}_{2}$-action on $\mathrm
{T}^{2\ast} M$ described in \reftext{Example~\ref{ex:Gk_Tk_starM}}. We shall
use standard coordinates $(p_{i}, p_{ij})$ on $\mathrm{T}^{2\ast
}M$.\footnote{If $[(f, x)] \sim(p_{i}, p_{ij})$ then $f(x)=0$ and
$f(x+h) = p_{i} h^{i} + \frac{1}{2} p_{ij}h^{i} h^{j} + o(|h|^{2})$, where $p_{ij}=p_{ji}$.}
We have
\[
(a, b).(p_{i}, p_{ij}) = (a p_{i}, a p_{ij}+b p_{i}\,p_{j}),
\]
so $\Delta= p_{i}\partial_{p_{i}} + p_{ij}\partial_{p_{ij}}$ and $Y
= p_{i}p_{j}\partial_{p_{ij}}$ has indeed weight $1$ with respect to
the standard vector bundle structure on $\mathrm{T}^{2\ast} M$.
\end{example}

Our goal now is to prove the inverse of \reftext{Lemma~\ref{lem:G2_actions_infinitesimally}}, i.e. to characterize right (resp.,
left) $\mathcal{G}_{2}$-actions in terms of a homogeneity structure
and a complete vector field $X$ of weight $-1$ (resp., $Y$ of weight +1).
Obviously, by our preliminary considerations knowing the actions of the
two canonical submonoids of $\mathcal{G}_{2}$ allows to determine the
action of the Lie subgroup ${\mathcal{G}}^{\text{inv}}_{2}$ of
invertible elements of $\mathcal{G}_{2}$. Yet problems with extending
this action on the whole $\mathcal{G}_{2}$ may appear.

\begin{lemma}
\label{lem:G_2_two}
Let $\tau:M\rightarrow M_{0}$ be a graded bundle (with the associated
weight vector field $\Delta$ and the corresponding homogeneity
structure $h:\mathbb{R}\times M\rightarrow M$) and let $X\in\mathfrak
{X}(M)$ (resp., $Y\in\mathfrak{X}(M)$) be a complete vector field of
weight $-1$ (resp., $+1$) i.e. $[\Delta,X]=-X$ (resp., $[\Delta, Y] =
Y$). Then the formulas
\begin{equation}\label{eqn:full_G2_actions}
p.(a,b):= X^{b/a}(h_{a}(p)) = h_{a}( X^{b/a^{2}}(p))\qquad
\left( \text{resp., }(a, b).p := h_{a}(Y^{b/a}(p))=
Y^{b/a^{2}}(h_{a}(p))\right)
\end{equation}
define a smooth right (resp., left) action of the group of invertible
elements $\mathcal{G}_{2}^{inv}\subset\mathcal{G}_{2}$. Here
$t\mapsto X^{t}$ (resp., $t\mapsto Y^{t}$) denotes the flow of the
vector field $X$ (resp., $Y$).
\end{lemma}

\begin{proof}
We shall restrict our attention to the case of the right action. The
reasoning for the case of the left action is analogous.

Note that the Lie algebra generated by vector fields $\Delta$ and $X$
is a non-trivial two-dimensional Lie algebra, thus it is isomorphic to
the Lie algebra $\mathfrak{aff}(\mathbb{R})$ of the Lie group
$$
\operatorname{Aff}(\mathbb{R})=\left\{
\begin{pmatrix} c & d\\ 0& 1
\end{pmatrix}
: c\neq0,\, d\in\mathbb{R}\right\}\subset\operatorname{Gl}_{2}(\mathbb{R})
$$
of affine transformations of $\mathbb{R}$. The identification
${\mathcal{G}}^{\text{inv}}_{2} \simeq\operatorname{Aff}(\mathbb
{R})$ is given by
\[
(a, b)\mapsto
\begin{pmatrix} 1/a & b/a^{2}\\ 0& 1
\end{pmatrix}
.
\]
Let $\operatorname{Aff}_{+}(\mathbb{R})$ be the subgroup of
orientation-preserving affine transformations of $\mathbb{R}$,
\[
\operatorname{Aff}_{+}(\mathbb{R}) = \left\{
\begin{pmatrix} c & d\\ 0& 1
\end{pmatrix}
: c>0,\, d\in\mathbb{R}\right\}.
\]
The subgroup $\operatorname{Aff}_{+}(\mathbb{R})$ is the connected
and simply-connected Lie group integrating
$\mathfrak{aff}(\mathbb{R})$. Clearly under the above identification
it is isomorphic to ${\mathcal{G}}^{\text{inv}}_{2+} = \{(a, b)\in
\mathcal{G}_{2}: a>0\}$. Thus, due to Palais' theorem and according to \reftext{\eqref{eqn:commutation_submonoids}},
formula \reftext{\eqref{eqn:full_G2_actions}} is a well-defined action of the
group ${\mathcal{G}}^{\text{inv}}_{2+}$ on $M$, i.e., both formulas
$p._{1}(a,b):=X^{b/a}(h_{a}(p))$ and $p._{2}(a,b):=h_{a}(
X^{b/a^{2}}(p))$ coincide for $(a,b)\in{\mathcal{G}}^{\text
{inv}}_{2+}$ and the resulting map is indeed a right action of
${\mathcal{G}}^{\text{inv}}_{2+}$.

Our goal now is to show that \reftext{\eqref{eqn:full_G2_actions}} is a
well-defined action of the whole ${\mathcal{G}}^{\text{inv}}_{2}$. It
is straightforward to check that $p._{1}(-1,0)=p._{2}(-1,0)=h_{-1}(p)$.
Now let us check that the formulas for $._{1}$ and $._{2}$ coincide for
elements of ${\mathcal{G}}^{\text{inv}}_{2}\setminus{\mathcal
{G}}^{\text{inv}}_{2+}$. Indeed, observe first that by \reftext{Definition~\ref{def:hgm_real}}, since $X$ is homogeneous of
weight~$-1$, for every $a\in
\mathbb{R}$
\[
(h_{a})_{\ast}X_{p}=a X_{h_{a}(p)}\ .
\]
Integrating the above equality we obtain the following result for flows:
\[
h_{a}(X^{t}(p))=X^{t a}(h_{a}(p))\ ,
\]
for every $a,t\in\mathbb{R}$.
Using this result we get for $a>0$
\[
p._{1}(-a,b)=X^{-b/a}(h_{-a}(p))=h_{-a}(X^{b/a^{2}}(p))=p._{2}(-a,b)\ ,
\]
i.e., \reftext{\eqref{eqn:full_G2_actions}} is well-defined on the whole
${\mathcal{G}}^{\text{inv}}_{2}$. To check that this is indeed an
action of ${\mathcal{G}}^{\text{inv}}_{2}$ note that
\[
p.(-a,b)=h_{-a}(X^{b/a^{2}}(p))=h_{-1}\left[
h_{a}(X^{b/a^{2}}(p))\right] =\left[ p.(a,b)\right] .(-1,0)
\]
and
\[
p.(-a,b)=X^{-b/a}(h_{-a}(p))=X^{-b/a}(h_{a}(h_{-1}(p)))=\left[
p.(-1,0)\right] .(a,-b)\ .
\]
In other words, the operation $p\mapsto p.(a,b)$ is compatible with the
following decomposition
%
\begin{equation}
\label{eqn:sd_product}
(-a,b)=(a,b)(-1, 0) =(-1, 0)(a,-b)\ .
\end{equation}
Now it suffices to observe that the latter formula allows to express
every multiplication of two elements in ${\mathcal{G}}^{\text
{inv}}_{2}$ as a composition of multiplications of elements of
${\mathcal{G}}^{\text{inv}}_{2+}$ and $(-1,0)$ (in other words,
${\mathcal{G}}^{\text{inv}}_{2}$ is a semi-direct product of
${\mathcal{G}}^{\text{inv}}_{2+}$ and $C_{2}\simeq\{(\pm1, 0)\}$).
Since formula \reftext{\eqref{eqn:full_G2_actions}} is multiplicative with
respect to ${\mathcal{G}}^{\text{inv}}_{2+}$ and $C_{2}$ and respects \reftext{\eqref{eqn:sd_product}}, it is a well-defined action of the whole
${\mathcal{G}}^{\text{inv}}_{2}$.
\end{proof}

We have thus shown that the infinitesimal data related with the right
(resp., left) action of $\mathcal{G}_{2}$ on a smooth manifold $M$,
i.e., a weight vector field $\Delta$ together with a complete vector
field $X$ (resp., $Y$) on $M$ satisfying \reftext{\eqref{eqn:commutator_X_Delta}}, integrates to a right (resp., left) action of
the Lie group ${\mathcal{G}}^{\text{inv}}_{2}$ on $M$. However, there
is no guarantee that this action will extend to the action of the whole
$\mathcal{G}_{2}\supset{\mathcal{G}}^{\text{inv}}_{2}$. This will
happen if formula \reftext{\eqref{eqn:full_G2_actions}} has a well-defined and
smooth extension to $a=0$. In particular situations this condition can
be checked by a direct calculation, yet no general criteria are known
to us.

In the forthcoming paragraphs we shall study (local) conditions of this
kind after restrict ourselves to the cases when the graded bundle $(M,
\Delta)$ is of low degree. The cases of left and right actions turned
out to be essentially different and so we treat them separately.

\paragraph*{Right $\mathcal{G}_{2}$-actions of degree at most~3}
Let us now classify (locally) all possible right $\mathcal
{G}_{2}$-actions on a smooth manifold $M$ such that the associated
graded bundle structure $(M,\Delta)$ is of degree at most 3. That is,
locally on $M$ we have graded coordinates
$(x^{i},y^{s}_{1},y_{2}^{S},y_{3}^{\sigma})$ where the lower index
indicates the degree.
By the results of \reftext{Remark~\ref{rem:euler_weight}}, the general formula
for a vector field $X$ of degree $-$1 in such a setting is
%
\begin{equation}
\label{eqn:field_weight_-1}
X=F^{s}(x)\partial_{y_{1}^{s}}+G^{S}_{s}(x) y_{1}^{s}\partial
_{y_{2}^{S}} + \left( H^{\sigma}_{S}(x) y^{S}_{2} + \frac{1}{2}
I^{\sigma}_{sr}(x) y_{1}^{s}y_{1}^{r}\right) \partial_{y_{3}^{\sigma
}}\ ,
\end{equation}
where $F^{s}$, $G^{S}_{s}$, $H^{\sigma}_{S}$, $I^{\sigma}_{sr}$ are
smooth functions on the base. The following result characterizes these
fields $X$ which give rise to a right action of the monoid $\mathcal
{G}_{2}$ on $M$:
%
\begin{lemma}\label{lem:right_g2_deg_less_4}
Let $(M,\Delta)$ be a graded bundle of degree at most 3 and let $X$ be
a weight $-1$ vector field on $M$ given locally by formula \reftext{\eqref{eqn:field_weight_-1}}. Then the right action $H:M\times{\mathcal
{G}}^{\text{inv}}_{2}\rightarrow M$ defined in \reftext{Lemma~\ref{lem:G_2_two}} extends to a smooth right action of $\mathcal{G}_{2}$ on
$M$ if and only if $F^{s}=0$ and $H^{\sigma}_{S} G^{S}_{s}=0$ for
every indices $s$ and $\sigma$. Equivalently,
$X$ is a degree~$-1$ vector field tangent to the fibration
$M=M^{3}\rightarrow M^{1}$, such that the differential weight $-2$
operator $X\circ X$ vanishes on all functions on $M$ of weight less or
equal $3$.
\end{lemma}

\begin{proof}
In order to find the flow $t\mapsto X^{t}$ we need to solve the
following system of ODEs
\begin{align*}
\dot{x}^{i}&=0\\
\dot{y}_{1}^{s}&= F^{s}(x)\\
\dot{y}_{2}^{S} &= G^{S}_{s}(x)y_{1}^{s}(t)\\
\dot{y}_{3}^{\sigma}&= H^{\sigma}_{S}(x) y_{2}^{S}(t)+\frac{1}{2}
I^{\sigma}_{sr}(x) y_{1}^{s}(t)y_{1}^{r}(t)\ ,
\end{align*}
which gives the following output
\begin{align*}
x^{i}(t)&=x^{i}(0),\\
y_{1}^{s}(t)&=y_{1}^{s}(0) + t F^{s},\\
y_{2}^{S}(t)&=y_{2}^{S}(0)+ tG^{S}_{s} y_{1}^{s}(0)+\frac{1}{2} t^{2}
G^{S}_{s}F^{s},\\
y_{3}^{\sigma}(t)&=y_{3}^{\sigma}(0)+ t (H^{\sigma}_{S}
y_{2}^{S}(0)+\frac{1}{2}I^{\sigma}_{sr}
y_{1}^{s}(0)y_{1}^{r}(0))+\frac{1}{2} t^{2}(H^{\sigma}_{S} G^{S}_{s}
y_{1}^{s}(0) + I^{\sigma}_{sr} F^{s} y_{1}^{r}(0)) \\
&\quad {}+ \frac{1}{6} t^{3} (H^{\sigma}_{S} G^{S}_{s}F^{s}+I^{\sigma}_{sr}
F^{s} F^{r})\ .
\end{align*}
Now, by \reftext{Lemma~\ref{lem:G_2_two}}, the action of $(a, b)\in\mathcal
{G}_{2}$ on $p\in M$ should be defined as $h_{a}\left(
X^{b/a^{2}}(p)\right) $, that is, it affects the coordinate
$y_{w}^{\alpha}$ of weight $w$ by $y_{w}^{\alpha}\mapsto a^{w}
y_{w}^{\alpha}(b/a^{2})$. Thus we have
\begin{align*}
H_{(a,b)}^{\ast}x^{i}&=x^{i},\\
H_{(a,b)}^{\ast}y_{1}^{s}&=a y_{1}^{s}+ \frac{b}{a} F^{s},\\
H_{(a,b)}^{\ast}y_{2}^{S}&=a^{2} y_{2}^{S}+ b G^{S}_{s}
y_{1}^{s}+\frac{1}{2} \frac{b^{2}}{a^{2}} G^{S}_{s} F^{s},\\
H_{(a,b)}^{\ast}y_{3}^{\alpha}&=a^{3} y_{3}^{\sigma}+ a b (H^{\sigma
}_{S} y_{2}^{S}+ I^{\sigma}_{sr} y_{1}^{s} y_{1}^{r})+\frac{1}{2}
\frac{b^{2}}{a}(H^{\alpha}_{S} G^{S}_{s} y_{1}^{s} + 2 I^{\alpha
}_{sr} F^{s} y_{1}^{r})+\frac{1}{6} \frac{b^{3}}{a^{3}} (H^{\sigma
}_{S} G^{S}_{s} F^{s} + I^{\alpha}_{sr} F^{s} F^{r})\, .
\end{align*}
Now it is clear that the action $H_{(a,b)}$ depends smoothly on $(a,b)$
if and only if $F^{s}=0$ and $H^{\sigma}_{S} G^{S}_{s}=0$.

The vector field $X$ is tangent to the fibration $M^{3}\rightarrow
M^{1}$ if and only if $F^{s}=0$. Then the condition on the differential
operator $X\circ X$ means that $0 = X(X(y^{\sigma}_{3})) =
X(X(y^{s}_{1} y^{S}_{2}))$ which simplifies to $H^{\sigma}_{S} G^{S}_{s}=0$.
\end{proof}

Note that in degree $1$ (i.e., when $(M, \Delta)$ is a vector bundle) $X$ has to be the zero vector field, hence
$
v.(a, b) = a\cdot v
$
for any $(a, b)\in\mathcal{G}_2$ and $v\in M$. 

In degree $2$  the only possibility is $X=G^S_s y_1^s\partial_{y_2^S}$. In geometric terms, $(G^S_s)$ defines a vector bundle morphism
$$
\phi: F^1 \rightarrow \widehat{F^2}, \quad \phi(x, y^s_1) = (x, \widehat{y}^S_2 = G^S_s(x) y^s_1),
$$
covering $\operatorname{id}_{M_0}$ where $M=F^2\rightarrow F^1\rightarrow M_0$ is the tower of affine bundle projections \eqref{eqn:tower_grd_spaces} associated with $(M, \Delta)$. Thus,  there is a one-to-one correspondence between degree $2$ right $\mathcal{G}_2$-actions and vector bundle morphisms $\phi$ as above. The correspondence is defined by the formula 
$$
v.(a, b) = h_a(v) + b\, \phi(v),
$$
where $v\in F^2$; $h_a$, for $a\in \mathbb{R}$, are homotheties of $(F^2, \Delta)$; and 
$
+ : F^2\times_{M_0}\widehat{F^2} \rightarrow F^2
$  
is the canonical action of the core bundle on a graded bundle.
Obviously in higher degrees finding the precise conditions for $X$ gets more complicated (yet is still doable in finite time) and more classes of admissible weight -1 vector fields appear.

\paragraph*{Left $\mathcal{G}_{2}$-actions}
In case of left $\mathcal{G}_{2}$-actions we meet a problem of
integrating a vector field $Y$ of weight 1. Even if the associated
graded bundle $(M, \Delta)$ is a vector bundle (a graded bundle of
degree 1), a vector field of weight 1 has a general form
\[
\frac{1}{2} F_{ij}^{k}(x) y^{i} y^{j} \partial_{y^{k}} + F^{a}_{i}(x)
y^{i} \partial_{x^{a}}
\]
and, in general, is not integrable in quadratures. Therefore, the
problem of classifying left $\mathcal{G}_{2}$-action seems to be more
difficult.
We will solve it in the simplest case when $(M, \Delta)$ is a vector bundle.

\begin{lemma}\label{lem:left_g2_action_vb}
Let $\tau:E\rightarrow M$ be a vector bundle. There is a one-to-one
correspondence between smooth left $\mathcal{G}_{2}$-actions on $E$
such that the multiplicative submonoid $(\mathbb{R},\cdot)\subset
\mathcal{G}_{2}$ acts by the homotheties of $E$ and symmetric
bi-linear operations $\bullet: E\times_{M} E\rightarrow E$ such that
for any $v\in E$
%
\begin{equation}\label{eqn:Identity}
v\bullet(v\bullet v) = 0\ .
\end{equation}
This correspondence is given by the following formula
%
\begin{equation}\label{eqn:left_G_2-action}
(a, b).v = a\,v + b\, v\bullet v\ ,
\end{equation}
where $(a,b)\in\mathcal{G}_{2}$ and $v\in E$.
\end{lemma}

\begin{proof}
We shall denote the action of an element $(a,b)\in\mathcal{G}_{2}$ on
$v\in E$ by $(a,b).v$. Observe first that we can restrict our attention
to a single fiber of $E$. Indeed, since $\tau(v)=(0,0).v$, we have
$\tau((a, b).v)=(0,0).(a, b).v = (0,0).v =\tau(v)$ and thus $(a,b).v$
belongs to the same fiber of $E$ as $v$ does. In consequence, without
any loss of generality, we may assume that $E$ is a vector
space.\smallskip

By the results of \reftext{Lemma~\ref{lem:G2_actions_infinitesimally}}, every
left $\mathcal{G}_{2}$-action on $E$ induces a weight 1 homogeneous
vector field $Y\in\mathfrak{X}(E)$. The flow of such a $Y$ at time
$t$ corresponds to the action of an element $(1,t)\in\mathcal{G}_{2}$.

Choose now a basis $\{e_{i}\}_{i\in I}$ of $E$ and denote by $\{y^{i}\}
_{i\in I}$ the related linear coordinates. In this setting (cf.~\reftext{Remark~\ref{rem:euler_weight}}) $Y$ writes as
\[
Y = \frac{1}{2} F_{ij}^{k}\ y^{i} y^{j} \partial_{y^{k}}, \quad\text
{where}\quad F_{ij}^{k} = F_{ji}^{k}\ .
\]
Let us now define the product $\bullet$ on base elements of $E$ by the formula
\[
e_{i} \bullet e_{j} = F_{ij}^{k} e_{k}\ ,
\]
and extend it be-linearly to an operation $\bullet:E\times
_{M}E\rightarrow E$. In other words, $Y(v)=v\bullet v$, where we use
the canonical identification of the vertical tangent vectors of
$\mathrm{T}E$ with elements of $E$.
\smallskip

We shall now show that the action of $\mathcal{G}_{2}$ is given by
formula \reftext{\eqref{eqn:left_G_2-action}}. Recall that by \reftext{Lemma~\ref{lem:G_2_two}} the action of $(a,b)\in{\mathcal{G}}^{\text{inv}}_{2}$
on $v\in E$ is given by
%
\begin{equation}\label{eqn:G_2_inv_on_E}
(a, b).v = a\cdot v(b/a)\ ,
\end{equation}
where $t\mapsto v(t):=(1,t).v$ denotes the integral curve of $Y$
emerging from $v(0)=v$ at $t=0$. The question is whether the above
formula extends smoothly to the whole $\mathcal{G}_{2}$.

Note that for $t\neq0$ we have $(t,tb).v=t v(b)$. Thus, assuming the
existence of a smooth extension of \reftext{\eqref{eqn:G_2_inv_on_E}} to the
whole $\mathcal{G}_{2}$, we have
\[
\left.\frac{\mathrm{d}^{}}{\mathrm{d}t^{}}\right|_{t=0} (t,
tb).v=v(b)\ .
\]
On the other hand, by the Leibniz rule we can write
\begin{align*}
\left.\frac{\mathrm{d}^{}}{\mathrm{d}t^{}}\right|_{t=0} (t,
tb).v&=\left.\frac{\mathrm{d}^{}}{\mathrm{d}t^{}}\right|_{t=0} (t,
0).v + \left.\frac{\mathrm{d}^{}}{\mathrm{d}t^{}}\right|_{t=0} (0,
t b).v=\left.\frac{\mathrm{d}^{}}{\mathrm{d}t^{}}\right|_{t=0} (t,
0).v + \left.\frac{\mathrm{d}^{}}{\mathrm{d}t^{}}\right|_{t=0} (t
b,0).(0,1).v\\
&=\left.\frac{\mathrm{d}^{}}{\mathrm{d}t^{}}\right|_{t=0} t\cdot v+
\left.\frac{\mathrm{d}^{}}{\mathrm{d}t^{}}\right|_{t=0} tb \cdot
(0,1).v= v+b\cdot(0,1).v\ .
\end{align*}
We conclude that for every $b\in\mathbb{R}$
%
\begin{equation}\label{eqn:action_G2}
v(b)=v+b\cdot(0,1).v\ ,
\end{equation}
i.e., integral curves of $Y$ are straight lines or constant curves.
Differentiating the above formula with respect to $b$ we get
$(0,1).v=Y(v)=v\bullet v$. Using this and \reftext{\eqref{eqn:G_2_inv_on_E}} we
get for $(a,b)\in{\mathcal{G}}^{\text{inv}}_{2}$
\[
(a,b).v=a\cdot v(b/a)=a\left( v+b/a\cdot v\bullet v\right) =a\cdot
v+b\cdot v\bullet v\ .
\]
Clearly this formula extends smoothly to the whole $\mathcal{G}_{2}$.
We have thus proved that any smooth $\mathcal{G}_{2}$-action on $E$
such that $(\mathbb{R},\cdot)\subset\mathcal{G}_{2}$ acts by the
homotheties of $E$ is of the form \reftext{\eqref{eqn:left_G_2-action}}.

Clearly formula \reftext{\eqref{eqn:left_G_2-action}} considered for some, a
priori arbitrary, bi-linear operation $\bullet$ defines a left
$\mathcal{G}_{2}$-action if and only if $(aA, aB+bA^{2}).v\overset{\text{\reftext{\eqref{eqn:G_2_multiplication}}}}{=}(a,b)(A,B).v$ for every $v\in E$
and $(a,b),(A,B)\in\mathcal{G}_{2}$. It is a matter of a simple
calculation to check that this requirement leads to the following condition:
\[
2 AbB v\bullet(v\bullet v) + bB^{2} (v\bullet v)\bullet(v\bullet v)=0.
\]
Since $v$ and $a$, $b$, $A$ and $B$ were arbitrary this is equivalent
to $v\bullet(v\bullet v)=0$ and $(v\bullet v)\bullet(v\bullet v)=0$
for every $v\in E$. To end the proof it amounts to show that this
latter condition is induced by the former one. Indeed, after a short
calculation formula \reftext{\eqref{eqn:Identity}} considered for $v=v'+t\cdot
w$ leads to the following condition
\[
t\left[ w\bullet(v'\bullet v')+2v'\bullet(v'\bullet w)\right]
+t^{2}\left[ v'\bullet(w\bullet w)+2w\bullet(w\bullet v')\right]
=0\ .
\]
Thus, as $t\in\mathbb{R}$ was arbitrary, $w\bullet(v'\bullet
v')+2v'\bullet(v'\bullet w)=0$ for every $v',w\in E$. In particular,
taking $w=v'\bullet v'$ and using \reftext{\eqref{eqn:Identity}} we get
$(v'\bullet v')\bullet(v'\bullet v')=0$. This ends the proof. 
\end{proof}

\paragraph*{Actions of the monoid of 2 by 2 matrices}
Let $\mathcal{G}:= \operatorname{M}_{2\times2}(\mathbb{R})$ be the
monoid of $2$ by $2$ matrices with the natural matrix multiplication.
We shall end our considerations in this section by studying smooth
actions of this structure on manifolds.

We have a canonical isomorphism $\mathcal{G}\simeq\mathcal{G}^{\text
{op}}$ which sends a matrix to its transpose. Thus, unlike the case of
$\mathcal{G}_{2}$, left and right $\mathcal{G}$-actions are in
one-to-one correspondence.
Moreover, any $\mathcal{G}$-action gives rise to left and right
$\mathcal{G}_{2}$-action as there is a canonical
monoid embedding $\mathcal{G}_{2}\rightarrow\mathcal{G}$, $(a,
b)\mapsto\left(
\begin{array}{cc}
a & b \\
0 & a^{2}
\end{array}
\right) $. This observation allows to prove easily the following result.

\begin{lemma}\label{lem:matrix} Any $\mathcal{G}$-action on a
manifold $M$ gives rise to a double graded bundle $(M, \Delta_{1},
\Delta_{2})$ equipped with two complete vector fields $X$, $Y$ of
weights $(1, -1)$ and $(-1, 1)$ respectively, such that $[X, Y] =
\Delta^{1}-\Delta^{2}$.
\end{lemma}

\begin{proof} Since the homogeneity structures defined by the actions of
the submonoids $G_{1} = \{\operatorname{diag}(t, 1):t\in\mathbb{R}\}
$, $G_{2}=\{\operatorname{diag}(1, t):t\in\mathbb{R}\}$,
$G_{1}\simeq(\mathbb{R}, \cdot)\simeq G_{2}$, commute, the
corresponding weight vector fields $\Delta_{1}$, $\Delta_{2}$ also
commute and give rise to a double graded structure $(M, \Delta_{1},
\Delta_{2})$. Define vector fields $X$, $Y$ as infinitesimal actions
of the subgroups $ \left(
\begin{array}{cc}
1 & t \\
0 & 1
\end{array}
\right) $ and $\left(
\begin{array}{cc}
1 & 0 \\
t & 1
\end{array}
\right) $, respectively. It is straightforward to check, as we did for
$\mathcal{G}_{2}$-actions, that $[\Delta_{1}, X] = - X$, $[\Delta
_{2}, X] = X$, so $X$ has weight $w(X)=(-1,1)$. Similarly, $[\Delta
_{1}, Y] = Y$, $[\Delta_{2}, Y] = -Y$, so $w(Y)=(1,-1)$, and moreover
\[
Y^{-s}\circ X^{-t}\circ Y^{s} \circ X^{t} =
\begin{pmatrix}
1-ts & -t^{2}s\\
s^{2}t & 1 + st + o(st)
\end{pmatrix}
\]
so $[X, Y] = \Delta_{2}-\Delta_{1}$ as we claimed.
\end{proof}
%
\begin{example} Let $M$ be a manifold and consider the space $\mathrm
{J}^{2}_{(0,0)}(\mathbb{R}\times\mathbb{R}, M)$ of all 2nd-jets at $(0,0)$ of maps $\gamma: \mathbb{R}^{2} \rightarrow
M$. Given local coordinates $(x^{j})$ on $M$, the adapted local
coordinates $(x_{00}^{j}, x_{10}^{j}, x_{01}^{j}, x_{20}^{j},
x_{11}^{j}, x_{02}^{j})$ on $\mathrm{J}^{k}_{(0,0)}(\mathbb{R}^{2},
M)$ of $[\gamma]_{2}$ are defined as coefficients of the Taylor
expansion
%
\begin{equation}
\gamma(t,s) = (\gamma^{j}(t,s)), \quad\gamma^{j}(t,s) = x_{00}^{j}
+ x_{10}^{j} t + x_{01}^{j} s +
x_{20}^{j} \frac{t^{2}}{2} + x_{11}^{j} ts + x_{02}^{j} \frac
{s^{2}}{2} + o(t^{2}, ts, s^{2}).
\end{equation}
The right action of $A\in\mathcal{G}$ on $[\gamma]_{2}\in M$ equals
$[\gamma(at+bs, ct+ds)]_{2}$ and reads as
%
\begin{align}
(x_{00}^{j},  x_{10}^{j}, x_{01}^{j}, x_{20}^{j}, x_{11}^{j}, x_{02}^{j}).
\begin{pmatrix}
a& b\\ c& d
\end{pmatrix}
&= (x_{00}^{j}, a x_{10}^{j} + c x_{01}^{j}, b x_{10}^{j} +
d x_{01}^{j}, a^{2} x_{20}^{j} + 2 ac x_{11}^{j} + c^{2} x_{02}^{j}, ab x_{20}^{j}
\nonumber\\
&\quad {}
+
(ad+bc) x_{11}^{j} + cd x_{02}^{j}, b^{2} x_{20}^{j} + 2bd x_{11}^{j} +
d^{2} x_{02}^{j})
\end{align}
Hence the action of $
\begin{pmatrix}1 & t \\ 0 &1
\end{pmatrix}
$ yields a vector field $X = x_{10}^{j} \partial_{x_{01}^{j}}
+{x_{20}^{j}}\,\partial_{x_{11}^{j}} + 2\,x_{11}^{j}\partial
_{x_{02}^{j}}$ of weight $(1, -1)$. Similarly, the action of $
\begin{pmatrix}1 & 0 \\ t &1
\end{pmatrix}
$ gives rise to a vector field $Y = x_{01}^{j} \partial_{x_{10}^{j}} +
x_{02}^{j} \partial_{x_{11}^{j}} + 2 x_{11}^{j} \partial
_{x_{20}^{j}}$ of weight $(-1, 1)$. We have
\[
[X, Y] = x_{10}^{j}\partial_{x_{10}^{j}} - x_{01}^{j}\partial
_{x_{01}^{j}} + 2 x_{20}^{j}\partial_{x_{20}^{j}} -2
x_{02}^{j}\partial_{x_{02}^{j}} = \Delta_{1} - \Delta_{2},
\]
where $\Delta_{1} = x_{10}^{j}\partial_{x_{10}^{j}} + 2
x_{20}^{j}\partial_{x_{20}^{j}}+ x_{11}^{j}\partial_{x_{11}^{j}}$,
$\Delta_{2} = x_{01}^{j}\partial_{x_{01}^{j}}+x_{11}^{j} \partial
_{x_{11}^{j}} + 2 x_{02}^{j}\partial_{x_{02}^{j}}$ are commuting
weight vector fields. This example has a direct generalization for the
case of higher order $(1, 1)$-velocities.
\end{example}

\section{On actions of the monoid of real numbers on supermanifolds}
\label{sec:super}

\paragraph*{Super graded bundles}
The notions of a super vector bundle (see e.g. \cite{BCC_sVB_2011})
and a graded bundle generalize naturally to the notion of a \emph
{super graded bundle}, i.e., a graded bundle in the category of
supermanifolds. The latter is a \emph{super fiber bundle} (see e.g.,
\cite{BCC_sVB_2011})
$\pi: \mathcal{E}\rightarrow\mathcal{M}$ in which one can
distinguish a class of $\mathbb{N}$-graded fiber coordinates so that
transition functions preserve this gradation (\reftext{Definition~\ref{def:s_grd_bndl}}).
On the other hand, super graded bundles are a particular example of
non-negatively graded manifolds in the sense of Voronov \cite
{Voronov:2001qf}. These are defined as supermanifolds with a privileged
class of atlases in which one assigns $\mathbb{N}_{0}$-weights to
particular coordinates. Coordinates of positive weights are
`cylindrical' and coordinate changes are decreed to be polynomials
which preserve $\mathbb{Z}_{2}\times\mathbb{N}_{0}$-gradation. The
coordinate parity is not determined by its $\mathbb{N}_{0}$-weight.

Our goal in this section is to prove a direct analog of \reftext{Theorem~\ref{thm:eqiv_real}} in supergeometry: $(\mathbb{R}, \cdot)$-actions on
supermanifolds are in one-to-one correspondence with super graded bundles.

To fix the notation, given a supermanifold defined by its structure
sheaf $(M, \mathcal{O}_{\mathcal{M}})$, we shall usually denote it
shortly by $\mathcal{M}$. Here $M := |\mathcal{M}|$ is a topological
space called the \emph{body of $\mathcal{M}$}. Elements of $\mathcal
{O}_{\mathcal{M}}(U)$ will be called \emph{local functions} on
$\mathcal{M}$. For an open subset $U$ of $M$ let $\mathcal
{J}_{\mathcal{M}}(U)$ be the ideal of nilpotent elements in $\mathcal
{O}_{\mathcal{M}}(U)$. The quotient sheaf $\mathcal{O}_{\mathcal
{M}}/\mathcal{J}_{\mathcal{M}}$ defines a structure of a real smooth
manifold on the body $|\mathcal{M}|$. For local functions $f, g\in
\mathcal{O}_{\mathcal{M}}(U)$ a formula $f=g + o(\mathcal
{J}_{\mathcal{M}}^{i})$ means that $f-g\in(\mathcal{J}_{\mathcal
{M}}(U))^{i}$.

The definition of a super graded bundle, alike its classical analog,
will be given in steps. We begin by introducing the notion of a super
graded space, which is, basically speaking, a superdomain $\mathbb
{R}^{m|n}$ equipped with an atlas of global graded coordinates.

\begin{definition}\label{def:s_grd_space}
Let $\mathbf{d}:= (\mathbf{d}_{\bar{0}}|\mathbf{d}_{\bar{1}})$,
where $\mathbf{d}_{\varepsilon} = (d_{\varepsilon,1}, \ldots,
d_{\varepsilon,k})$, $\varepsilon\in\mathbb{Z}_{2}$, $1\leq i\leq
k$ are sequences of non-negative integers, and let $|\mathbf
{d}_{\varepsilon}|:=\sum_{i=1}^{k} d_{\varepsilon, i}$.

A \emph{super graded space of rank $\mathbf{d}$} is a supermanifold
$\mathrm{W}$ isomorphic\footnote{In the context of supermanifolds we
should rather speak about isomorphism than diffeomorphism. Since there
is no concept of a topological supermanifold, there is no need to
distinguish between homeomorphisms, $C^{1}$-diffeomorphism or smooth
diffeomorphisms.} to a superdomain $\mathbb{R}^{|\mathbf{d}_{\bar
{0}}| \big\vert \vert \mathbf{d}_{\bar{1}}|}$ and equipped with an
equivalence class of graded coordinates.
Here we assume that the number of even (resp. odd) coordinates of
weight $i$ is equal to $d_{\bar{0}, i}$ (resp. $d_{\bar{1}, i}$) where
$1\leq i\leq k$.
Two systems of graded coordinates are \emph{equivalent} if they are
related by a polynomial transformation with coefficients in $\mathbb
{R}$ which preserve both the parity and the weights (cf. \reftext{Definition~\ref{def:gr_space}}).

A \emph{morphism} between super graded spaces $W_{1}$ and $W_{2}$ is a
map $\Phi: W_{1}\rightarrow W_{2}$ which in some (and thus any) graded
coordinates writes as a polynomial respecting the $\mathbb
{N}_{0}\times\mathbb{Z}_{2}$-gradation.
\end{definition}

Informally speaking, a super graded bundle is a collection of super
graded spaces parametrized
by a base supermanifold.

\begin{definition}\label{def:s_grd_bndl}
A \emph{super graded bundle of rank $\mathbf{d}$} is a super fiber
bundle $\pi:\mathcal{E}\rightarrow\mathcal{M}$ with the typical
fiber $\mathbb{R}^{\mathbf{d}}$ considered as a super graded space of
rank $\mathbf{d}$. In other words, there exists a cover $\{\mathcal
{U}_{i}\}$ of the supermanifold $\mathcal{M}$ such that the total
space $\mathcal{E}$ is obtained by gluing trivial super graded bundles
$\mathcal{U}_{i}\times\mathbb{R}^{\mathbf{d}}\rightarrow\mathcal
{U}_{i}$ by means of transformations $\phi_{ij}: \mathcal
{U}_{ij}\times\mathbb{R}^{\mathbf{d}}\rightarrow\mathcal
{U}_{ij}\times\mathbb{R}^{\mathbf{d}}$ of the form
%
\begin{eqnarray*}[cc]
y^{a'} &= \sum_{I, J} Q^{a'}_{I, J}(x, \theta) y^{a_{1}}\ldots
y^{a_{i}} \xi^{A_{1}}\ldots\xi^{A_{j}}, \\
\xi^{A'} &= \sum_{I, J} Q^{A'}_{I, J}(x, \theta) y^{a_{1}}\ldots
y^{a_{i}} \xi^{A_{1}}\ldots\xi^{A_{j}}\ .
\end{eqnarray*}
Here $\mathcal{U}_{ij}:= \mathcal{U}_{i}\cap\mathcal{U}_{j}$;
$Q^{a'}_{I, J}$ and $Q^{A'}_{I, J}$ are local functions of (super)
coordinate functions $(x^{i}, \theta^{\alpha})$ on $\mathcal
{U}_{ij}$; $(y^{a},\xi^{A})$ and $(y^{a'},\xi^{A'})$ are graded super
coordinates on fibers $\mathbb{R}^{\mathbf{d}}$; and the summation is
over such sets of indices $I=(a_{1}, \ldots, a_{i})$ and $J= (A_{1},
\ldots, A_{j})$ that the parity and the weight of each monomial in the
sums on the right coincides with the parity and the weight of the
corresponding coordinate on the left.

The notion of a \emph{morphism} $\Phi: \mathcal{E}\rightarrow
\mathcal{E}'$ between super graded bundles is clear; it is enough to
assume that $\Phi$ is a morphism of supermanifolds such that the
corresponding algebra map $\Phi^{*}: \mathcal{O}_{\mathcal
{E}'}(|\mathcal{E}'|) \rightarrow\mathcal{O}_{\mathcal
{E}}(|\mathcal{E}|)$ preserves the $\mathbb{N}_{0}$-gradation.
\end{definition}

\begin{example} Higher tangent bundles have their analogs in supergeometry.
Given a supermanifold $\mathcal{M}$ a higher tangent bundle $\mathrm
{T}^{k} \mathcal{M}$ is a natural example of a super graded
bundle.\footnote{In a natural way higher tangent bundles correspond to
the (purely even) Weil algebra $\mathbb{R}[\varepsilon]/\langle
\varepsilon^{k+1}\rangle$ (see \cite
{Kolar_Michor_Slovak_nat_oper_diff_geom_1993}). In general, any super
Weil algebra gives rise to a Weil functor that can be applied to a
supermanifold (cf. \cite{BF_Weil_supermanifolds}).} For $k=2$ and
local coordinates $(x^{A})$ on $\mathcal{M}$ (even or odd) one can
introduce natural coordinates $(x^{A},\dot{x}^{B}, \ddot{x}^{C})$ on
$\mathrm{T}^{2}\mathcal{M}$ where coordinates $\dot{x}^{A}$ and
$\ddot{x}^{A}$ share the same parity as $x^{A}$ and are of weight 1
and 2, respectively. Standard transformation rules apply:
\[
x^{A'}= x^{A'}(x), \quad\dot{x}^{A'} = \dot{x}^{B}\frac{\partial
x^{A'}}{\partial x^{B}}, \quad\ddot{x}^{A'} = \ddot{x}^{B}\frac
{\partial x^{A'}}{\partial x^{B}} + \dot{x}^{C}\dot{x}^{B} \frac
{\partial^{2} x^{A'}}{\partial x^{B}\partial x^{C}}.
\]
\end{example}

\paragraph*{Homogeneity structures in the category of supermanifolds}
Also the notion of a homogeneity structure easily generalizes to the
setting of supergeometry.

\begin{definition}\label{def:super_hgm_structure}
A \emph{homogeneity structure on a supermanifold $\mathcal{M}$} is a
smooth action $h:\mathbb{R}\times\mathcal{M}\to\mathcal{M}$ of the
multiplicative monoid $(\mathbb{R}, \cdot)$ of
real numbers, i.e., $h$ is a morphism of supermanifolds such that the
following diagram
\[
\xymatrix{\mathbb{R}\times\mathbb{R}\times\mathcal{M}\ar
[d]_{m\times\operatorname{id}_\mathcal{M}} \ar[rr]^{\operatorname
{id}_\mathbb{R}\times h} && \mathbb{R}\times\mathcal{M}\ar[d]^h\\
\mathbb{R}\times\mathcal{M}\ar[rr]^h && \mathcal{M}
}
\]
commutes (here $m:\mathbb{R}\times\mathbb{R}\rightarrow\mathbb{R}$
denotes the standard multiplication) and that $h_{1} := h|_{\{1\}\times
\mathcal{M}}: \mathcal{M}\rightarrow\mathcal{M}$ is the identity
morphism. In other words,
$h$ is a morphism of supermanifolds defined by a collection of maps
$h_{t}:\mathcal{M}\to\mathcal{M}$, $t\in\mathbb{R}$ such that
$h_{ts}^{*} = h_{t}^{*}\circ h_{s}^{*}$ for any $t, s\in\mathbb{R}$
and that $h_{1} = \operatorname{id}_{\mathcal{M}}$.

A \emph{morphism} of two homogeneity structures $(\mathcal
{M}_{1},h_{1})$ and $(\mathcal{M}_{2},h_{2})$ is a morphism $\Phi
:\mathcal{M}_{1}\rightarrow\mathcal{M}_{2}$ of supermanifolds
intertwining the actions $h_{1}$ and $h_{2}$. Clearly, homogeneity
structures on supermanifolds with their morphism form a \emph{category}.

Per analogy to the standard (real) case, we say that a local function
$f\in\mathcal{O}_{\mathcal{M}}(U)$, where $U$ is an open subset of $|M|$,
is called \emph{homogeneous} of \emph{weight} $w\in\mathbb{N}$ if
\[
h_{t}^{*}(f) = t^{w}\cdot f,
\]
for any $t\in\mathbb{R}$. We assume here that the carrier $U$ of $f$
is preserved by the action $h$, i.e. $\underline{h_{t}}(U)\subset U$
for any $t\in\mathbb{R}$.
\end{definition}

\begin{remmark}\label{rem:super_hgm_structure}
Observe that given a homogeneity structure $h$ on a supermanifold
$\mathcal{M}$ the induced maps $\underline{h_{t}}: |\mathcal
{M}|\rightarrow|\mathcal{M}|$ equip the body $|\mathcal{M}|$ with a
(standard) \emph{homogeneity structure}, and so $\underline{h_{0}}:
|\mathcal{M}|\rightarrow|\mathcal{M}|_{0}$ is a (real) \emph{graded
bundle} over $|\mathcal{M}|_{0}:= h_{0}(|\mathcal{M}|)$.

Note also that, analogously to the standard case, every super graded
bundle structure $\pi:\mathcal{E}\rightarrow\mathcal{M}$ provides
$\mathcal{E}$ with a \emph{canonical homogeneity structure}
$h^{\mathcal{E}}$ defined locally in an obvious way. We call it an
action of the \emph{homotheties of $\mathcal{E}$}. Obviously, a
morphism of super graded bundles $\Phi:\mathcal{E}_{1}\rightarrow
\mathcal{E}_{2}$ induces a morphism of the related homogeneity
structures $(\mathcal{E}_{1},h^{\mathcal{E}_{1}})$ and $(\mathcal
{E}_{2},h^{\mathcal{E}_{2}})$.
\end{remmark}

Alike in the standard case, homogeneous functions on super graded
bundles are polynomial in graded coordinates.

\begin{lemma}\label{lem:super_morphisms}
Let $f$ be a homogeneous function on a trivial super graded bundle
$\mathcal{U}\times\mathbb{R}^{\mathbf{d}}$, where $\mathcal{U}$ is
a superdomain and $\mathbf{d}= (\mathbf{d}_{\bar{0}}|\mathbf
{d}_{\bar{1}})$ is as above. Then $f$ is a homogeneous polynomial in
graded fiber coordinates.
\end{lemma}

\begin{proof} This follows directly from a corresponding result for purely
even graded bundles. Indeed, let $f\in\mathcal{C}^{\infty}(x,
y^{a}_{w})[\theta, \xi^{A}]$ be an even or odd, homogeneous function
on $\mathcal{U}\times\mathbb{R}^{\mathbf{d}}$:
\[
f=\sum_{I, J} f_{I, J}(x, y^{a}) \xi^{A_{1}} \ldots\xi^{A_{i}}
\theta^{B_{1}} \ldots\theta^{B_{j}},
\]
where $(y^{a}, \xi^{A})$ are graded coordinates on $\mathbb
{R}^{\mathbf{d}}$, and the summation goes over sequences
$I = \{A_{1}<\ldots<A_{i}\}$, $J = \{B_{1}<\ldots<B_{j}\}$. Then
\[
h_{t}^{*}f = \sum_{I, J} f_{I, J}(x, t^{\mathbf{w}(a)} y^{a})
t^{\mathbf{w}(A_{1})+\ldots+\mathbf{w}(A_{i})} \xi^{A_{1}} \ldots
\xi^{A_{i}} \theta^{B_{1}} \ldots\theta^{B_{j}},
\]
so $h_{t}^{*}f = t^{w} f$ implies that the coefficients $f_{I, J}$ are
real functions of weight
$w - (\mathbf{w}(A_{1})+\ldots\mathbf{w}(A_{i}))\geq0$, thus
polynomials in $y^{a}$.
\end{proof}

In what follows we will need the following technical result which
allows to construct homogeneous coordinates under mild technical conditions.
%
\begin{lemma}\label{lem:better_coord_for_h} Consider a superdomain
$\mathcal{M}= U\times\Pi\mathbb{R}^{s}$ with $U\subset\mathbb
{R}^{r}$ being an open set and introduce super coordinates
$(y^{1},\ldots,y^{r},\xi^{i},\ldots,\xi^{s})$ on $\mathcal{M}$,
i.e. $y$'s are even and $\xi$'s are odd coordinates on $\mathcal{M}$.
Let $h$ be an action of the monoid $(\mathbb{R}, \cdot)$ on $\mathcal
{M}$ such that
\[
h_{t}^{*}(y^{a})= t^{\mathbf{w}(a)}\,y^{a} + o(\mathcal{J}_{\mathcal
{M}}), \quad\text{and}\quad h_{t}^{*}(\xi^{i}) = t^{\mathbf{w}(i)}\,
\xi^{i} + o(\mathcal{J}_{\mathcal{M}}^{2})\ .
\]
Then
\[
\left( \frac{1}{\mathbf{w}(a)!} \left.\frac{\mathrm{d}^{\mathbf
{w}(a)}}{\mathrm{d}t^{\mathbf{w}(a)}}\right|_{t=0} h_{t}^{*}(y^{a}),
\frac{1}{\mathbf{w}(i)!} \left.\frac{\mathrm{d}^{\mathbf
{w}(i)}}{\mathrm{d}t^{\mathbf{w}(i)}}\right|_{t=0} h_{t}^{*}(\xi
^{i})\right)
\]
are graded coordinates on the superdomain $\mathcal{M}$.
\end{lemma}
\begin{proof} We remark that if
\[
h_{t}^{*} f = \sum_{I\subset\{1, \ldots, s\}} g_{I}(t, y^{1}, \ldots
, y^{r}) \xi^{I} \in\mathcal{C}^{\infty}(\mathbb{R}\times U)[\xi
^{1}, \ldots, \xi^{s}]
\]
is a function on $\mathcal{M}$, then the function $f^{[k]}:=\frac
{1}{k!} \left.\frac{\mathrm{d}^k}{\mathrm{d}t^k}\right|_{t=0}
h_{t}^{*} f$ is well-defined as $h$ is smooth and is given by
\[
f^{[k]} = \frac{1}{k!} \sum_{I\subset\{1, \ldots, s\}}\xi^{I}\,
\left.\frac{\mathrm{d}^k}{\mathrm{d}t^k}\right|_{t=0} g_{I}(t,
y^{1}, \ldots, y^{r}) \in\mathcal{C}^{\infty}(U)[\xi^{1}, \ldots,
\xi^{s}].
\]
Now, since for any morphism $\phi:\mathcal{M}\rightarrow\mathcal
{M}$, $(\operatorname{id}_{\mathbb{R}}\times\phi)^{*}: \mathcal
{O}_{\mathbb{R}\times\mathcal{M}}\rightarrow\mathcal{O}_{\mathbb
{R}\times\mathcal{M}}$ commutes with the operators $\left.\frac
{\mathrm{d}^k}{\mathrm{d}t^k}\right|_{t=0}: \mathcal{O}_{\mathbb
{R}\times\mathcal{M}}\rightarrow\mathcal{O}_{\mathbb{R}\times
\mathcal{M}}$, we get
\[
h_{s}^{*} f^{[k]} = \frac{1}{k!} \left.\frac{\mathrm{d}^k}{\mathrm
{d}t^k}\right|_{t=0} h^{*}_{s} h^{*}_{t} f = \frac{1}{k!} \left
.\frac{\mathrm{d}^k}{\mathrm{d}t^k}\right|_{t=0} h^{*}_{st} f =
s^{k} \left.\frac{\mathrm{d}^k}{\mathrm{d}t^k}\right|_{t=0}
h^{*}_{t} f = s^{k}\, f^{[k]}\ ,
\]
that is, $f^{[k]}$ is $h$-homogeneous of weight $k$. In particular,
\[
(y^{a})^{[\mathbf{w}(a)]} = y^{a} + o(\mathcal{J}_{\mathcal{M}})
\quad\text{and}\quad(\xi^{i})^{[\mathbf{w}(i)]} = \xi^{i} +
o(\mathcal{J}_{\mathcal{M}}^{2}),
\]
are homogeneous with respect to $h$. To prove that these are true
coordinates on $\mathcal{M}$ observe that the matrices $\left( \frac
{\partial(y^{a})^{[\mathbf{w}(a)]}}{\partial{y^{b}}}\right) $ and
$\left( \frac{\partial(\xi^{i})^{[\mathbf{w}(i)]}}{\partial\xi
^{j}}\right) $ are invertible, so the result follows. 
\end{proof}

\paragraph*{The main result} We are now ready to prove that \reftext{Theorem~\ref{thm:eqiv_real}} generalizes to the supergeometric context.

\begin{theorem}\label{thm:main_super}
The categories of super graded bundles (with connected bodies) and
homogeneity structures on supermanifolds (with connected bodies) are
equivalent. At the level of objects this equivalence is provided by the
following two constructions
\begin{itemize}
\item With every super graded bundle $\pi:\mathcal{E}\rightarrow
\mathcal{M}$ one can associate a homogeneity structure $(\mathcal
{E},h^{\mathcal{E}})$, where $h^{\mathcal{E}}$ is the action by the
homotheties of $\mathcal{E}$.
\item Given a homogeneity structure $(\mathcal{M},h)$ on a
supermanifold $\mathcal{M}$, the map $h_{0}:\mathcal{M}\rightarrow
\mathcal{M}_{0}:=h_{0}(\mathcal{M})$ provides $\mathcal{M}$ with a
canonical structure of a super graded bundle such that $h$ is the related action by homotheties.
\end{itemize}
At the level of morphisms: every morphism of super graded bundles is a
morphism of the related homogeneity structures and, conversely, every
morphism of homogeneity structures on supermanifolds respects the
canonical super graded bundle structures.
\end{theorem}

\begin{proof}
The crucial part of the proof is to show that given a homogeneity
structure $h:\mathbb{R}\times\mathcal{M}\rightarrow\mathcal{M}$ on
a supermanifold $\mathcal{M}$ one can always find an atlas with
homogeneous coordinates on $\mathcal{M}$. First we observe that without
loss of generality we may assume that $\mathcal{M}$ has a simple form,
namely $\mathcal{M}$ is isomorphic with $U\times\mathbb{R}^{\mathbf
{d}}\times\Pi\mathbb{R}^{q}$ for some small open subset 
$U\subset\mathbb{R}^{n}$, i.e. $\mathcal{M}$ has a second, other than $h$,
homogeneity structure associated with a vector bundle $E= U\times
\mathbb{R}^{\mathbf{d}}\times\mathbb{R}^{q}\to U\times\mathbb
{R}^{\mathbf{d}}$. Using the fact that these graded bundle structures
are compatible, and transferring the homogeneity structure $h$ to the
real manifold $E$ (with some loss of information) we are able to
construct graded coordinates for $\mathcal{M}$ but modulo $\mathcal
{J}_{\mathcal{M}}^{2}$. Then we evoke \reftext{Lemma~\ref{lem:better_coord_for_h}} to finish the proof.

Assume that $h:\mathbb{R}\times\mathcal{M}\rightarrow\mathcal{M}$
is a homogeneity structure on a supermanifold $\mathcal{M}$.
Recall (see \reftext{Remark~\ref{rem:super_hgm_structure}}) that $h$ induces a
canonical homogeneity structure $\underline{h}$ on the body $|\mathcal{M}|$.
Since we work locally we may assume without any loss of generality that
$|\mathcal{M}|_{0}:= \underline{h}_{0}(|\mathcal{M}|)$ is an open
contractible subset $U\subset\mathbb{R}^{n}$, and $|\mathcal{M}| =
U\times\mathbb{R}^{\mathbf{d}}$ is a trivial graded bundle over $U$
of rank $\mathbf{d}=(d_{1}, \ldots, d_{k})$. Thus
we may assume that $\mathcal{M}= \Pi E$ where $E=U\times\mathbb
{R}^{\mathbf{d}}\times\mathbb{R}^{q}$ is a trivial vector bundle
over $|\mathcal{M}| = U\times\mathbb{R}^{\mathbf{d}}$ with the
typical fiber $\mathbb{R}^{q}$. Note, that we do not need to refer to
Batchelor's theorem \cite{Gaw_77,Batchelor_str_sMnflds} and
the argument works even for holomorphic actions of the monoid $(\mathbb
{C}, \cdot)$ on complex supermanifolds (see \reftext{Remark~\ref{rem:complex_sMnflds}}).

Consider now local coordinates $(x^{i}, y^{a}_{w}, Y^{A})$ on $E$ where
$(x^{i},y^{a}_{w})$ are graded coordinates on the base and $Y^{A}$ are
linear coordinates on fibers. Let $(\xi^{A})$ be odd coordinates on
$\Pi E$ corresponding to $(Y^{A})$.
Recall that $\mathcal{J}_{\mathcal{M}}(|\mathcal{M}|)=\langle\xi
^{A}\rangle$ denotes
the nilpotent radical of $\mathcal{O}_{\mathcal{M}}(|\mathcal{M}|)$.
Since $(x^{i},y^{a}_{w})$ are graded coordinates with respect to
$\underline{h}$ and since $h_{t}$ respects the parity for each $t\in
\mathbb{R}$, the general form of $h_{t}$ must be
%
\begin{equation}\label{eqn:h_t_form}
\begin{cases}
h_{t}^{*}(x^{i}) &= x^{i} + o(\mathcal{J}_{\mathcal{M}}^{2}),\\
h_{t}^{*}(y^{a}_{w}) &= t^{w} \,y^{a}_{w} + o(\mathcal{J}_{\mathcal
{M}}^{2}),\\
h_{t}^{*}(\xi^{A}) &= \alpha^{A}_{B}(t,x^{i},y^{a}_{w}) \xi^{B} +
o(\mathcal{J}_{\mathcal{M}}^{2}),
\end{cases}
\end{equation}
where $\alpha^{A}_{B}$ are smooth functions.

The action $h$ defines an action $\widetilde{h}$ of the monoid
$(\mathbb{R}, \cdot)$ on $E$ which is given by
%
\begin{equation}\label{eqn:action_h_tilde}
\widetilde{h}_{t}^{*}(x^{i})=x^{i}, \quad\widetilde
{h}_{t}^{*}(y^{a})= t^{w(y^{a})} \,y^{a}, \quad\text{and}\quad
\widetilde{h}_{t}^{*}(Y^{A}) = \alpha^{A}_{B}(t,x,y) Y^{B}.
\end{equation}
Indeed, by reducing $h_{t}^{*}: \mathcal{O}_{\mathcal{M}}(|\mathcal
{M}|) \rightarrow\mathcal{O}_{\mathcal{M}}(|\mathcal{M}|)$ modulo
$\mathcal{J}_{\mathcal{M}}^{2}(|\mathcal{M}|)$ we obtain an
endomorphism of $\mathcal{O}_{\mathcal{M}}(|\mathcal{M}|)/\mathcal
{J}_{\mathcal{M}}^{2}(|M|) = \mathcal{C}^{\infty}(x^{i}, y^{a}_{w})
\oplus\xi^{A} \cdot\mathcal{C}^{\infty}(x^{i}, y^{a}_{w}) \simeq
\mathcal{C}^{\infty}(x^{i}, y^{a}_{w}) \oplus Y^{A} \cdot\mathcal
{C}^{\infty}(x^{i}, y^{a}_{w})$, thus $\widetilde{h}_{t}\circ
\widetilde{h}_{s} = \widetilde{h}_{ts}$ and $\widetilde{h}_{t}$ does
not depend on a particular choice of linear coordinates $Y^{A}$ on $E$.
It follows from \reftext{Theorem~\ref{thm:eqiv_real}} that $E$ is a graded
bundle over $E_{0}:=\widetilde{h}_{0}(E)$, whose homotheties coincide
with the maps $\widetilde{h}_{t}$. Note that the inclusions $U\times\{
0\} \times\{0\} \subset E_{0} \subset U\times\{0\}\times\mathbb
{R}^{q}$ can be proper.

Our goal now is to find graded coordinates on $E$ out of non-homogeneous
coordinates $(x^{i}, y^{a}, Y^{A})$ and then mimic the same changes of
coordinates in order to define a graded coordinate system on the
supermanifold $\mathcal{M}= \Pi E$ out of a non-homogeneous one
$(x^{i}, y^{a}, \xi^{A})$.

Denote by $H$ the homotheties related with the vector bundle structure
on $\tau: E \rightarrow|\mathcal{M}|=U\times\mathbb{R}^{\mathbf
{d}}$. A~fundamental observation that follows from \reftext{\eqref{eqn:action_h_tilde}} is that the actions $H$ and $\widetilde{h}$
commute, i.e.,
\[
H_{s} \circ\widetilde{h}_{t} = \widetilde{h}_{t}\circ H_{s}
\]
for every $t,s\in\mathbb{R}$.
Thus $(E, \widetilde{h}, H)$ is a \emph{double homogeneity structure}
and, by Theorem 5.1 of \cite{JG_MR_gr_bund_hgm_str_2011}, a double
graded bundle:
\[
\xymatrix{
E=U\times\mathbb{R}^{\mathbf{d}}\times\mathbb{R}^q \ar
[d]^{\widetilde{h}_0}\ar[rr]^{H_0} && U\times\mathbb{R}^{\mathbf{d}}
\ar[d]^{\widetilde{h}_0|_{U\times\mathbb{R}^{\mathbf{d}}}} \\
E_0 \ar[rr]^{H_0|_{E_0}} && U.
}
\]
Moreover, the above-mentioned result implies that we can complete
graded coordinates $(x^{i}, y^{a}_{w})$ on $U\times\mathbb
{R}^{\mathbf{d}}$ which are constant along fibers of the projection
$H_{0}$ with graded coordinates $\tilde{Y}^{A}_{w}$ of bi-weight $(w,
1)$, where $0\leq w\leq k$ so that $(x^{i}, y^{a}_{w}, \widetilde
{Y}^{A}_{w})$ is a system of bi-graded coordinates for $(E, \widetilde
{h}, H)$.

Since both $(Y^{A})$ and $(\widetilde{Y}^{A}_{w})$ are linear
coordinates for the vector bundle $H_{0}:E\rightarrow U\times\mathbb
{R}^{\mathbf{d}}$ they are related~by
%
\begin{equation}\label{eqn:tilde_Y_A}
\widetilde{Y}^{A}_{w} = \gamma^{A}_{B}(x, y) Y^{B}
\end{equation}
for some functions $\gamma^{A}_{B}$ on $U\times\mathbb{R}^{d}$.

Let us define $\widetilde{\xi}^{A}:= \gamma^{A}_{B}(x,y)\cdot\xi
^{B}$, i.e. using the same functions as in \reftext{\eqref{eqn:tilde_Y_A}}. By
applying $\widetilde{h}_{t}^{*}$ to \reftext{\eqref{eqn:tilde_Y_A}} we get
\[
t^{\mathbf{w}(A)} \gamma^{A}_{C}(x, y)Y^{C}
\overset{\text{\reftext{\eqref{eqn:tilde_Y_A}}}}{=} t^{\mathbf{w}(A)} \widetilde{Y}_{w}^{A} =
\widetilde{h}_{t}^{\ast}(\widetilde{Y}_{w}^{A}) \overset{\text{\reftext{\eqref{eqn:tilde_Y_A}}}}{=}\widetilde{h}_{t}^{*}(\gamma^{A}_{B}(x, y))
\widetilde{h}_{t}^{*}(Y^{B}) \overset{\text{\reftext{\eqref{eqn:action_h_tilde}}}}{=}
\gamma^{A}_{B}(x, t^{w} y^{a}_{w})\alpha^{B}_{C}(t, x, y) Y^{C},
\]
hence
$t^{\mathbf{w}(A)} \gamma^{A}_{C} = \gamma^{A}_{B}(x, t^{w}
y^{a}_{w})\alpha^{B}_{C}(t, x, y)$, and
%
\begin{equation}
\begin{aligned}[c]
h_{t}^{*}(\widetilde{\xi}^{A})&=h_{t}^{*}(\gamma
^{A}_{B})h_{t}^{*}(\xi^{B}) \overset{\text{\reftext{\eqref{eqn:h_t_form}}}}{=} \gamma
^{A}_{B}(x, t^{w} y^{a}_{w} + o(\mathcal{J}_{\mathcal
{M}}^{2}))(\alpha^{B}_{C}(t, x, y) \xi^{C} + o(\mathcal{J}_{\mathcal
{M}}^{2})) = \\
&= \gamma^{A}_{B}(x, t^{w} y^{a}_{w}) \alpha^{B}_{C}(t, x, y)\xi^{C}
+ o(\mathcal{J}_{\mathcal{M}}^{2}) = t^{\mathbf{w}(A)} \gamma
^{A}_{C}(x, y) \xi^{C} + o(\mathcal{J}_{\mathcal{M}}^{2}) = \\
&= t^{\mathbf{w}(A)} \widetilde{\xi}^{A} + o(\mathcal{J}_{\mathcal
{M}}^{2}).
\end{aligned}
\end{equation}
We obtain a graded coordinate system for $\mathcal{M}$ due to
\reftext{Lemma~\ref{lem:better_coord_for_h}}.

The equivalence at the level of morphism follows directly from \reftext{Lemma~\ref{lem:super_morphisms}}. This result implies that locally any
supermanifold morphism respecting the homogeneity structures is a
homogeneous polynomial in graded coordinates, i.e. it is a morphism of
the related super graded bundles (cf. \reftext{Definition~\ref{def:s_grd_bndl}}).
\end{proof}

\begin{remmark}\label{rem:complex_sMnflds} Using the same methods one
can prove an analog of above result for holomorphic supermanifolds (a
super-version of complex manifolds): a holomorphic action of $(\mathbb
{C},\cdot)$ on a holomorphic supermanifold $\mathcal{M}$ gives rise
to a graded holomorphic super coordinate system for $\mathcal{M}$.
Indeed, the proof of \reftext{Lemma~\ref{lem:better_coord_for_h}} can be
rewritten in a holomorphic setting. The other result we need to
complete the proof of \reftext{Theorem~\ref{thm:main_super}} in the holomorphic
context is that two holomorphic commuting $(\mathbb{C}, \cdot)$
actions on a complex manifold $M$ give rise to $\mathbb{N}_{0}\times
\mathbb{N}_{0}$ graded coordinate system on $M$ (an analog of Theorem~5.1
\cite{JG_MR_gr_bund_hgm_str_2011}). This can be justified using a
double graded version of \reftext{Lemma~\ref{lem:complex_gr_sspace}} and the
fact that $M$ can be considered as a substructure of $\mathrm
{J}^{r}\mathrm{J}^{s} M$ for some $r$, $s$ (a double holomorphic
homogeneity substructure). Details are left to the Reader.
\end{remmark}

\section*{Acknowledgments}
This research was supported by the {Polish National Science Centre} grant
under the contract number {DEC-2012/06/A/ST1/00256}.

The question of characterizing the actions of the multiplicative monoid
$(\mathbb{C},\cdot)$ occurred during the discussion between
Professors Stanis{\l}aw L. Woronowicz and Janusz Grabowski at the
seminar on the results of \cite{JG_MR_gr_bund_hgm_str_2011}. The
problem of characterizing $\mathcal{G}_{k}$-actions was originally
posted by Professor Janusz Grabowski. We would like to thank them for
the inspiration and encouragement to undertake this research.

\bibliographystyle{amsalpha}
\bibliography{biblMonoids}
\end{document}